\documentclass[12pt]{amsart}
\usepackage{a4wide}
\numberwithin{equation}{section}
\usepackage{mathrsfs}
\usepackage{amsmath,extpfeil}
\usepackage{amssymb}
\usepackage{amsthm}
\usepackage{mathrsfs}
\usepackage{indentfirst}
\usepackage{footnote}
\usepackage{enumerate}
\usepackage{amsmath,amsfonts,amssymb,amsthm}
\usepackage{appendix}
\usepackage[numbers,sort&compress]{natbib}
\usepackage{color}
\usepackage[colorlinks, linkcolor=blue, citecolor=blue]{hyperref}

\newtheorem{prop}{Proposition}[section]
\newtheorem{theorem}{Theorem}[section]
\newtheorem*{theorem*}{Theorem A}
\newtheorem{lemma}{Lemma}[section]

\newtheorem*{definition}{Definition}
\newtheorem{remark}{Remark}[section]    
\numberwithin{equation}{section}
\newtheorem{cor}[theorem]{Corollary}

\newcommand{\beq}{\begin{equation}}
\newcommand{\eeq}{\end{equation}}

\allowdisplaybreaks[4]
\begin{document}

\title[Global regularity for Monge-Ampere equation in polytopes]
{Global regularity for the Dirichlet problem of Monge-Amp\`ere equation in convex polytopes}
\author[G. Huang and W. Shen]
{Genggeng Huang and Weiming Shen}
\address[Genggeng Huang]
{School of Mathematical Sciences, Fudan University, Shanghai 200433, China.}
\email{genggenghuang@fudan.edu.cn}

\address[Weiming Shen]{School of Mathematical Sciences,
Capital Normal University,
Beijing, 100048, China}
\email{wmshen@aliyun.com}

\thanks{The first author is partially supported by NSFC 12141105. The second author is partially supported by NSFC 12371208.}

\subjclass[2020]{}

\date{}
\maketitle

\begin{abstract}We study the Dirichlet problem for Monge-Amp\`ere equation in bounded convex polytopes. We give sharp conditions  for the existence of global $C^2$ and $C^{2,\alpha}$ convex solutions  provided that
a global $C^2$, convex subsolution exists.
\end{abstract}

\section{Introduction}
Consider the following Monge-Amp\`ere equation with Dirichlet boundary condition
\begin{equation}\label{intro-1}
\begin{split}
\det D^2 u=&f(x),\quad \text{in}\quad \Omega,\\
u=&\varphi(x),\quad \text{on}\quad \partial  \Omega.
\end{split}
\end{equation}
There has been a long history for the study of Monge-Amp\`ere equation due to its important applications in geometry, fluid dynamics, finance mathematics and computor sciences.
\par The Dirichlet boundary problem of Monge-Amp\`ere equation is a classical problem. Usually, people call \eqref{intro-1} non-degenerate if  $f\ge c_0>0$ in $\Omega$.
On smooth and uniformly convex domains, under the assumptions $0<f\in C^{1,1}(\overline{\Omega})$ and $\varphi\in C^3(\overline{\Omega})$,  the global $C^{2,\alpha}$ estimates are obtained by Ivockina \cite{Iv1980}, Krylov \cite{Kry1983} and Caffarelli-Nirenberg-Spruck \cite{CaffarelliNirenbergSpruck1984}. One of the key step  is the derivation of the estimate of $\|u\|_{C^{1,1}(\overline\Omega)}$ in the above works where one needs to differentiate \eqref{intro-1} twice. Hence, the condition $f\in C^{1,1}(\overline{\Omega})$ cannot be removed which is different from the Laplace equations. It is an interesting and important question to formulate sharp conditions on the right-hand side and boundary data for global $C^{2,\alpha}$ estimates.  Wang \cite{Wang1996} relaxed the conditions to $f\in C^{0,1}(\overline{\Omega})$ and $\partial\Omega\in C^3$, $\varphi\in C^{3}(\overline{\Omega})$. He also constructed examples to show that $\partial\Omega\in C^3$ and $\varphi\in C^{3}(\overline{\Omega})$ are optimal. For uniformly convex domains, under the sharp conditions, global $C^{2,\alpha}$ estimates were obtained by Trudinger-Wang \cite{TW2008}. For smooth convex domains, under the sharp conditions, pointwise $C^{2,\alpha}$ estimates up to the boundary were proved by Savin \cite{Savin2013}.
\par An important assumption in the above work is that $\Omega$ is smooth and convex. Moreover, in the work of \cite{CaffarelliNirenbergSpruck1984}, the uniform convexity of the domain $\Omega$ is used to guarantee the existence of smooth sub-solutions of \eqref{intro-1}.  Hence, Guan-Spruck \cite{GuanSpruck1993} first obtained the global $C^{2,\alpha}$ estimate by assuming the existence of a convex strict subsolution $\underline{u}\in C^2(\overline{\Omega})$ for general smooth domain $\Omega$ without convexity. The strictness of subsolution $\underline u$ was removed by Guan \cite{Guan1998}.

The global $C^{2,\alpha}$ estimates for all the above works are under the assumptions that the domains $\Omega$ are smooth enough.  Recently, Le-Savin \cite{LeSavin2023} made a first attempt on the global $C^{2,\alpha}$ estimates in two dimensional polygons.
They proved the following result.
\begin{theorem*} [Theorem 1.1 \cite{LeSavin2023}]  Let $\Omega$ be a bounded convex polygonal domain in $\mathbb R^2$. Let $u$ be a convex function that solves the Dirichlet problem for the Monge-Amp\`ere equation \eqref{intro-1}. Assume that for some $\beta\in(0,1)$,
\begin{equation*}
f\in C^\beta(\overline{\Omega}),\quad f>0,\quad \varphi\in C^{2,\beta}(\partial\Omega),
\end{equation*}
and there exists a globally $C^2$, convex, strictly sub-solution $\underline u\in C^2(\overline{\Omega})$ to \eqref{intro-1}. Then, $u\in C^{2,\alpha}(\overline\Omega)$ for some $\alpha>0$. The constant $\alpha$ and the global $C^{2,\alpha}$ norm $\|u\|_{C^{2,\alpha}(\overline{\Omega})}$ depend on $\Omega$, $\beta$, $\min_{\overline\Omega}f$, $\|f\|_{C^\beta(\overline{\Omega})}$, $\|\varphi\|_{C^{2,\beta}(\partial\Omega)}$, $\|\underline u\|_{C^2(\overline\Omega)}$, and the differences $\det D^2{\underline u}-f$ at the vertices of $\Omega$.
\end{theorem*}
The results of global $C^{2,\alpha}$ regularity for the solutions of Monge-Amp\`ere equation are very few.
Jhaveri \cite{Jhaveri2019} investigated the the global $C^{2,\alpha}$ regularity for the Monge-Amp\`ere equation with the second boundary condition in the planar polygon domain and the high dimensional case was shown by Chen-Liu-Wang in \cite{ChenLiuWang2023}.
Polytopes also naturally arise in the Monge-Amp\`ere equation with Guillemin boundary conditions(\cite{Rubin2015,Huang2023,HuangShen2023}).

In the present paper, we are interested in the global $C^{2,\alpha}$ estimates for the Dirichlet problem of \eqref{intro-1} in polytope domains in all dimensions.  The proof of Theorem A used the complex analysis and the fact that $u_{12}$ solves a second order linear elliptic equation with no zero order terms. Hence, we can not apply the methods in \cite{LeSavin2023} to high dimensional case.
We need some new observations. Before starting our main results, we first introduce some notations.

\begin{itemize}
	\item[$\bullet$] $\mathbf{\Gamma_k(\Omega)}$: $k-$skeleton of $n-$polytope $\Omega$. By $k$-skeleton we mean the collection of faces of dimension at most $k$.
	\item[$\bullet$] {\bf Adjacent}: two $(n-1)-$faces $F,\tilde F$ of $\Omega$ are called adjacent if and only if $F\cap \tilde F$ contains an $(n-2)-$face.
	\item[$\bullet$] {$\mathbf{\mathbb V_k}$}: a cone which is isomorphic to $(\mathbb R_+)^k$.
	\item[$\bullet$] {$\mathbf{V_\mu}$}: $V_\mu=\left\{(x_1,x_2)|\arccos \left(\frac{x_1}{\sqrt{x_1^2+x_2^2}}\right)\in (0,\mu\pi),\quad x_2>0\right\}$.
	\item[$\bullet$] {\bf simple polytope}: A polytope $P\subset\mathbb R^n$ is  simple if and only if the tangent cone $V$ at the vertices of $P$ is equivalent to $(\mathbb R_+)^n$ after an affine transformation.
	\item[$\bullet$] {\bf simplicial polytope}: A polytope $P\subset\mathbb R^n$ is  simplicial if and only if all its $(n-1)$-faces are simplex in $\mathbb R^{n-1}$.
	\end{itemize}
\begin{definition}
[$\mathbf{\Theta(u,x_0,\Omega),\theta(u,x_0,\Omega)}$] Let $\Omega$ be a convex $n-$polytope and $x_0\in \Gamma_{n-2}(\Omega)$. Suppose $u\in C^2(\overline{\Omega})$ is a uniformly convex function. Consider an affine transformation $T$ such that $v(x)=u(Tx)$ satisfies $D^2 v(T^{-1}x_0)=I_{n\times n}$. Let $V$ be the tangent cone of $T^{-1}(\Omega)$ at $T^{-1}(x_0)$. Then $\mathbf{\Theta(u,x_0,\Omega)(\theta(u,x_0,\Omega))}$ is defined as the maximum(minimum) dihedral angle between the adjacent $(n-1)-$faces of $V$.
\end{definition}

\begin{definition}[{\bf A-Condition}]
Let $u\in C^2(\overline{\Omega})$ be a uniformly convex function for some convex $n-$polytope $\Omega$.
For $x_0\in  \Gamma_{n-2}(\Omega)$, we say that $u$ satisfies the (Strong) A-condition at $x_0$ if and only if $\Theta(u,x_0,\Omega)(<) \leq \frac{\pi}{2}$.
\end{definition}

\begin{remark}
A-condition imposes strong restriction on the polytopes. It is pointed out that the polytope $\Omega$ must be a simple polytope in Proposition \ref{prop5.1:label}.
\end{remark}

\begin{theorem}\label{mainthm1}
Let $\Omega$  be a bounded convex $n-$polytope. Let $u$ be a convex function which solves \eqref{intro-1}.  Assume that for some $\beta\in(0,1)$,
\begin{equation}\label{cond-f-varphi}
f\in C^\beta(\overline{\Omega}),\quad f>0,\quad \varphi\in C^{2,\beta}(\overline \Omega),
\end{equation}
and there exists a globally $C^2$, convex, sub-solution $\underline u\in C^2(\overline{\Omega})$ to \eqref{intro-1}(that is $\det D^2\underline u\ge f$ in $\overline{\Omega}$ and $\underline u=\varphi$ on $\partial \Omega$). Then, $u\in C^{1,1}(\overline \Omega)\bigcap C^{2}(\overline\Omega\backslash \Gamma_{n-3}(\Omega))$. If in addition, the boundary data $\varphi$ satisfies the following two conditions.
\begin{itemize}
	\item[(C1).]  $\det D^2 \varphi=f$  on $\Gamma_{n-3}(\Omega)$.
	\item[(C2).]  $\varphi$ satisfies the $A-$condition on $\Gamma_{n-3}(\Omega)$.
\end{itemize}
Then $u\in C^2(\overline{\Omega})$.
\end{theorem}

\begin{remark}
Condition (C1) is a necessary compatibility condition. Condition (C2) is a necessary condition in the following sense.
If (C2) is violated, then \eqref{intro-1} does not admit any solution $u\in C^2(\overline{\Omega})$ provided that $\det D^2 \underline u\not\equiv f$ in $\Omega$(See Lemma \ref{cd2-pC2-1} and Corollary \ref{cor-notsimple-3d}).
\end{remark}

Theorem \ref{mainthm1} only proved that $u\in C^2(\overline{\Omega})$. However, $D^2u$ may have oscillation of order $1$ near $\Gamma_{n-2}(\Omega)$ even all the quantities in Theorem \ref{mainthm1} are uniformly bounded(See Remark 5.2 in \cite{LeSavin2023} for $n=2$).
In order to control the modulus of the continuity of $D^2u$, we have the following theorem.
\begin{theorem}\label{mainthm2}
Suppose all the conditions in Theorem \ref{mainthm1} are fulfilled and the (C2) condition is replaced by the following one.
\begin{itemize}
	\item[(C3).]  $\varphi$ satisfies the strong $A-$condition on $\Gamma_{n-3}(\Omega)$.
\end{itemize}
Moreover, if one of the following two conditions is satisfied.
\begin{itemize}
		\item[(C4).] $\det D^2{\underline u}(x)>f(x)$ on $\Gamma_{n-2}(\Omega)\backslash \Gamma_{n-3}(\Omega)$.
		\item[(C5).] $\underline{u}$ satisfies the strong $A-$condition on $\Gamma_{n-2}(\Omega)\backslash \Gamma_{n-3}(\Omega)$.
	\end{itemize}
 Then, $u\in C^{2,\alpha}(\overline \Omega)$ for some $\alpha\in (0,\beta)$. Moreover, the $C^{2,\alpha}-$norm of $u$ and $\alpha$ depend only on $\Omega$, $\beta$, $\|1/f\|_{L^\infty(\Omega)}$, $\|f\|_{C^{\beta}(\overline{\Omega})}$, $\|\varphi\|_{C^{2,\beta}(\partial \Omega)}$, $\|\underline{u}\|_{C^2(\overline{\Omega})}$ and $\varepsilon_0>0$ where
 \begin{equation*}
 \varepsilon_0=\min_{x\in \Gamma_{n-2}(\Omega)}\left(\frac \pi 2-\Theta(\underline u,x,\Omega)\right)
 \end{equation*}
 or
 \begin{equation*}
\begin{split}
  \varepsilon_0=\min \bigg\{ &
  \min_{x\in \Gamma_{n-3}(\Omega)}\left(\frac \pi 2-\Theta(\underline u,x,\Omega)\right), \\
  & \inf_{x\in \Gamma_{n-2}(\Omega)\backslash \Gamma_{n-3}(\Omega)}\left( \max\{\det D^2{\underline u}(x)-f(x),\frac \pi 2-\Theta(\underline u,x,\Omega)\} \right)\bigg\}.
\end{split}
\end{equation*}
\end{theorem}

Conditions (C4) or (C5) are used to control the modulus of the continuity of $D^2u$ near $\Gamma_{n-2}(\Omega)\backslash \Gamma_{n-3}(\Omega)$(See Proposition \ref{ modulusofcontinuity1} and Proposition \ref{ modulusofcontinuity2}). Condition (C3) is used to control the modulus of the continuity of $D^2u$ near $\Gamma_{n-3}(\Omega)$(See Lemma \ref{exm-sta:label} for the sharpness of condition (C3)).

 The existence of a $C^2$ sub-solution $\underline u$ is crucial in Theorem \ref{mainthm1} and Theorem \ref{mainthm2}. However, in general, it is rather difficult to construct a $C^2$ sub-solution $\underline u$ for given data $f,\varphi$.  When $\Omega$ is smooth and uniformly convex, then any  $\varphi\in C^\infty(\partial\Omega)$ can be extended to a uniformly convex function $\varphi\in C^\infty(\overline{\Omega})$ with $\det D^2 \varphi\ge f$.  When $\Omega$ is just a convex polytope, it is  non-trivial to extend a boundary uniformly convex function $\varphi\in C^\infty(\partial\Omega)$ to a uniformly convex function $\varphi\in C^2(\overline{\Omega})$.

 The present paper is organized as follows. In Section 2, we establish the global $C^{1,1}$-estimates which depends only on the existence of $C^2$ sub-solution $\underline u$. In Section 3, we mainly discuss various kinds of Liouville type theorems for Monge-Ampere equation in the cone. In Section 4-5, Theorem \ref{mainthm1} is proved by blow up argument which relies on the Liouville type theorems in Section 3. Also, we prove the existence of $C^2(\overline\Omega)$ function $u$ satisfying A-condition implies $\Omega$ is simple. The converse is also true for $n=2,3$(Proposition \ref{prop5.1:label}). In Section 6, Theorem \ref{mainthm2} is proved by rescaling arguments. Moreover, we provide  sufficient conditions on the existence of $C^{2,\alpha}(\overline \Omega)$ solutions without sub-solution conditions in $n=2,3$(Theorem \ref{thm6.3}).

\section{Global $C^{1,1}$ estimate}
As pointed out in the introduction, the global $C^{1,1}$ estimate of the solutions for \eqref{intro-1} was important for global $C^{2,\alpha}$ estimate in smooth domains $\Omega$. In this section, we prove that the global $C^{1,1}$ estimate of the solutions for \eqref{intro-1} holds even for polytope domains $\Omega$ provided   that   there exist globally $C^2$ convex sub-solutions.
\begin{lemma}\label{lem-C1}
Let $\Omega$  be a bounded convex $n-$polytope and $f,\varphi$ satisfy \eqref{cond-f-varphi}. Suppose  that there exists a globally $C^2$, convex, sub-solution $\underline u\in C^2(\overline{\Omega})$ to \eqref{intro-1}. Then,
\eqref{intro-1} admits a unique Alexandrov solution $u\in C^{0,1}(\overline{\Omega})\cap C^{2,\beta}(\overline{\Omega}\backslash \Gamma_{n-2}(\Omega))$.
\end{lemma}
\begin{proof}
The existence of Alexandrov solution $u\in C(\overline{\Omega})$ follows from [Theorem  8.2.7,\cite{Han2016book}]. Also from [Theorem 1.1,\cite{Savin2013}], $u\in C^{2,\beta}$ near $\Gamma_{n-1}(\Omega)\backslash \Gamma_{n-2}(\Omega)$. For any point $x_0\in \Gamma_{n-2}(\Omega)$, by subtracting a linear function, for $x$ near $x_0$, one knows
\begin{equation*}
\underline{u}(x)=O(|x-x_0|^2).
\end{equation*}
Since $\underline{u}(x)$ is a sub-solution and $u$ is convex, one obtains
\begin{equation*}
u(x)=O(|x-x_0|^2),\text{ for } x \text{ near } x_0.
\end{equation*}
This implies $u$ is differentiable on $\partial\Omega$.
\par Claim: $u$ is strictly convex in $\Omega$.
\\ Suppose $u$ is not strictly convex at $0\in \Omega$. Without loss of generality, we may assume $u(0)=0$ and $u\ge 0$. By [Theorem 1, \cite{Caffarelli1990-1}], $\{u=0\}$ contains a line segment whose end points $p,q$ are on $\partial\Omega$.  By [Theorem 1.1,\cite{Savin2013}], one knows that $p,q$ are not on $\Gamma_{n-1}(\Omega)\backslash\Gamma_{n-2}(\Omega)$. Now we can assume $p,q\in \Gamma_{n-2}(\Omega)$.
Notice  that
\begin{equation*}
\underline{u}(p+t(q-p))-\underline u(p)-t\nabla \underline{u}(p)\cdot (q-p)>0,\quad t\in (0,1).
\end{equation*}
Hence, $\nabla \underline{u}(p)\cdot (q-p)<0$.  This implies
\begin{equation*}
O(t^2)\ge u(p+t(q-p))-u(p)-t\nabla {u}(p)\cdot (q-p)=\big(-\nabla \underline{u}(p)\cdot (q-p)\big)t
\end{equation*}
This yields a contradiction and proves the claim.
\par
Then, $u\in C^{2,\beta}(\overline{\Omega}\backslash \Gamma_{n-2}(\Omega))$ follows from the interior estimates(Theorem 2,\cite{Caffarelli1990-2}) and boundary estimates(Theorem 1.1,\cite{Savin2013}).
\end{proof}

\begin{lemma}\label{lem-C11}
Suppose all the assumptions in Lemma \ref{lem-C1} are fulfilled. Let $u$ be a convex function which solves \eqref{intro-1}. Then, we have
$u\in C^{1,1}(\overline{\Omega})$ where the $C^{1,1}$-norm of $u$ depends  only on $\Omega$, $\beta$, $\|1/f\|_{L^\infty(\Omega)}$, $\|f\|_{C^{\beta}(\overline{\Omega})}$, $\|\varphi\|_{C^{2,\beta}(\partial \Omega)}$, $\|\underline{u}\|_{C^{1,1}(\overline{\Omega})}$.
\end{lemma}
\begin{proof}
From Lemma \ref{lem-C1}, we only need to consider the points near $\Gamma_{n-2}(\Omega)$.
Let $\sigma_0\in (0,1)$ be a small positive constant.  Suppose $0\in \Gamma_{n-2}(\Omega)\backslash \Gamma_{n-3}(\Omega)$. Let $x_0\in \Omega$ be a point satisfying
\begin{equation}\label{c11-pt1}
d(x_0,0)=d(x_0,\Gamma_{n-2}(\Omega)),\quad \frac{d(x_0,\Gamma_{n-2}(\Omega))}{d(x_0,\Gamma_{n-3}(\Omega))}\le \sigma_0.
\end{equation}
After subtracting a linear function and performing an affine transformation, we may assume
\begin{equation*}
u(0)=|\nabla u(0)|=0,\quad V=(\mathbb R_+)^2\times \mathbb R^{n-2}
\end{equation*}
where $V$ is the tangent cone of $\Omega$ at $0$.
Define
\begin{equation*}
u_\lambda(x)=\frac{u(\lambda x)}{\lambda^2},\quad {\underline u}_\lambda(x)=\frac{{\underline u}(\lambda x)}{\lambda^2},\quad \Omega_\lambda=\lambda^{-1}\Omega,\quad  \lambda=|x_0|.
\end{equation*}
Then, by  convexity, $u_\lambda\rightarrow u_\infty$ locally uniformly in $V$ as $\lambda\rightarrow 0$.

By gradient estimates for convex functions, we know
\begin{equation*}
|D_{x_i}u_\lambda|\le C, \quad \text{in}\quad [0,4]^2\times [-4,4]^{n-2},\quad i=3,\cdots,n.
\end{equation*}
Also, from the convexity, for $x\in [0,4]^2\times [-4,4]^{n-2}$, we know
\begin{equation*}
C\ge \frac{\partial u_\lambda}{\partial x_1}(4,x_2,x'')\ge \frac{\partial u_\lambda}{\partial x_1}(x_1,x_2,x'')\ge \frac{\partial u_\lambda}{\partial x_1}(0,x_2,x'')\ge  \frac{\partial \underline{u}_\lambda}{\partial x_1}(0,x_2,x'')\ge -C.
\end{equation*}
Hence, we can get
\begin{equation*}
|D_x u_\lambda|\le C,\quad \forall x\in [0,4]^2\times [-4,4]^{n-2}.
\end{equation*}

Claim: For any $K\subset\subset \Omega_\lambda$, the strict convexity of $u_\lambda$ is independent of $\lambda$.\\
Suppose not. Consider the point $x_\infty$ and the supporting function $l_\infty(x)$ of $u_\infty$ at $x_\infty$ such that
$\{u_\infty=l_\infty\}$ is not a single point.
\begin{itemize}
	\item[1.] $\{u_\infty=l_\infty\}$ contains a line or a ray. However, along the line or the ray, $u_\infty\ge \underline u_\infty$ where $\underline u_\infty$ grows quadratically near infinity. This is impossible.
	\item[2.] $\{u_\infty=l_\infty\}$ contains a line segment and  its end points lie on $\Gamma_{n-1}(V)$. We can use the same argument as in Lemma \ref{lem-C1} to exclude this case.
\end{itemize}

Hence, by the interior $C^{2,\beta}$ regularity(Theorem 2,\cite{Caffarelli1990-2}) and the boundary $C^{2,\beta}$ regularity(Theorem 1.1,\cite{Savin2013}), there holds
 \begin{equation*}
 \|u_{\lambda}\|_{C^{2,\beta}\left(\big( [0,2]^{2}\setminus[0,1/2)^2\big)\times [-2,2]^{n-2}\right)}\le C
 \end{equation*}
 for some positive constant $C$ independent of $\lambda$. Since $\frac{x_0}{|x_0|}\in  \big( [0,2]^{2}\setminus[0,1/2)^2\big)\times [-2,2]^{n-2}$, we obtain that
 \begin{equation*}
 |D^2 u(x)|\le C
 \end{equation*}
 for any point $x$ satisfying \eqref{c11-pt1}.

\par  For $x_0\in \Omega$ near $0\in \Gamma_{k}(\Omega)\backslash \Gamma_{k-1}(\Omega)$, $0\le k\le n-2$. We can use the method of induction and repeat the above argument to get the desired $C^{1,1}$-estimate.
\end{proof}

\begin{remark} The proof of Lemma \ref{lem-C11} only needs the existence of a $C^{1,1}-$sub-solution $\underline u$. This implies we don't need the compatibility condition (C1) for global $C^{1,1}$-estimates. Moreover, there is no restriction on the polytopes.
\end{remark}

\section{Solutions in the cones}
The $C^2$ regularity of the solutions of \eqref{intro-1} in $n-$polytope $\Omega$ is related to the Liouville type theorem in the cone.
It is important to discuss some related properties  of solutions of the Monge-Amp\`ere equation in the cone.
\par
Let $\Sigma$ be a connected sub-domain in $\mathbb S^{n-1}$. Define cone $V_\Sigma$ in $\mathbb R^n$ as follows.
\begin{equation*}
V_\Sigma=\{tx|x\in \Sigma,\	t>0\}.
\end{equation*}
Let $-\Delta_\theta$ be the Laplace-Beltrami operator on $\mathbb S^{n-1}$.
Consider the following eigenvalue problem:
\begin{equation*}
\begin{split}
-\Delta_\theta \phi_\Sigma&=\lambda_1 (\Sigma)\phi_\Sigma\quad\text{in } \Sigma,\\
\phi_{\Sigma}& =0,\quad \text{on}\quad \partial\Sigma.
\end{split}
\end{equation*}
Here $\lambda_1 (\Sigma)>0$ is the first eigenvalue and  $\phi_\Sigma\in H_0^1(\Sigma)\cap C^\infty(\Sigma)$ is the corresponding eigenfunction.  We also choose
\begin{equation*}
\phi_\Sigma>0,\quad \text{in}\quad \Sigma,\quad \|\phi_{\Sigma}\|_{L^2(\Sigma)}=1.
\end{equation*}
Then, there holds
\begin{equation}\label{eq-engen}
\begin{split}
-\Delta(|x|^\mu \phi_\Sigma)=-[\mu(\mu+n-2)-\lambda_1 (\Sigma)]|x|^{\mu-2} \phi_\Sigma,\quad\text{in } V_\Sigma.
\end{split}
\end{equation}
The following Liouville type theorem is important for our $C^2-$estimates.
\begin{theorem}\label{thm-liou1}
Let $\Sigma,V_\Sigma$ be given as above with $\lambda_1(\Sigma)=\mu(\mu+n-2)$, $\mu>2$.
Suppose that $u\in  C(\overline{V_\Sigma})\cap C^2(V_\Sigma)$ is a convex solution of
\begin{equation}\label{liou-1}
\begin{split}
\det D^2 u  &=1,\quad \text{in } V_\Sigma,\\
u&=\frac{|x|^2}{2},\quad\text{on }\partial V_\Sigma.
\end{split}
\end{equation}
In addition, we assume the sub-solution condition
\begin{equation}\label{sub-sol-cone}
u(x)\ge \frac{|x|^2}{2},\quad\text{in}\quad V_\Sigma,
\end{equation}
and the asymptotic condition
\begin{equation}\label{liou-growth}
\lim_{x\rightarrow \infty}\frac{u(x)}{|x|^{\mu'}}=0,\quad \text{for some}\quad \mu'\in (2,\mu).
\end{equation}
Then,
\begin{equation*}
u(x)=\frac{|x|^2}{2}.
\end{equation*}

\end{theorem}
\begin{proof}
Let $\Sigma\subset\subset \Sigma_0\subset \mathbb S^{n-1}$ be such that $\lambda_1(\Sigma_0)=\mu_0(n-2+\mu_0)$ for some $\mu_0\in (\mu',\mu)$.
Then,
$$0<c_1 \leq\phi_{\Sigma'} \leq c_2<+\infty,\quad \text{in}\quad  \overline{\Sigma}.$$
Also, there holds
$$-\Delta\big(|x|^{\mu_0} \phi_{\Sigma_0})=0.$$
$\forall \varepsilon>0$, set
$$u_{\varepsilon}(x)=\frac{|x|^2}{2}+\varepsilon|x|^{\mu_0}\phi_{\Sigma_0}.$$
By \eqref{liou-growth}, $u_{\varepsilon}>u$ for $x\in V_\Sigma\bigcap \overline{B_{R_\varepsilon}(0)}^c$ when $R_\varepsilon$ is sufficiently large.

Next, we claim
\begin{equation}\label{u-epsilon}
u_{\varepsilon}\geq u\quad\text{in } V_\Sigma.\end{equation}
If \eqref{u-epsilon} is not true,  we can assume $u_{\varepsilon}(x)-u(x)$ attains its negative minimum at $\bar x$ in the interior of $V_\Sigma$. Then at $\bar x$,
\begin{equation*}
D^2 u_{\varepsilon}(\bar x)\ge D^2 u(\bar x),
\end{equation*}
i.e. $D^2 u_{\varepsilon}$ is convex near $\bar x$.
Consider the following two cases.
\begin{itemize}
	\item[1.] $(|x|^{\mu_0}\phi_{\Sigma_0})_{ii}(\bar x) \neq 0$ for some $i=1,\cdots,n$. Then,
\begin{align*}
\det D^2 u_{\varepsilon}(\bar x)&\leq \prod_{j=1}^nu_{\varepsilon, jj}(\bar x) \\
&=e^{\sum\limits_{j=1}^n \ln (1+\varepsilon(|x|^{\mu_0}\phi_{\Sigma_0})_{jj} ) }|_{x=\bar x}\\
&<e^{\Delta (\varepsilon|x|^{\mu_0}\phi_{\Sigma_0})}|_{x=\bar x}=1,
\end{align*}
which yields a contradiction.
\item[2.] $(|x|^{\mu_0}\phi_{\Sigma_0})_{ii}(\bar x) =0$, $i=1,...,n$. Since $\frac{\partial^2}{\partial r^2}(|x|^{\mu_0}\phi_{\Sigma_0})(\bar x)>0$, one gets
 \begin{equation*}
 D^2 (|x|^{\mu_0}\phi_{\Sigma_0})(\bar x) \neq O_{n\times n}.
 \end{equation*} Without loss of generality, we assume $(|x|^{\mu_0}\phi_{\Sigma_0})_{12}(\bar x) \neq 0$.
 Then,
\begin{align*}
\det D^2 u_{\varepsilon}(\bar x)&\leq (u_{\varepsilon, 11}(\bar x)u_{\varepsilon, 22}(\bar x)-u_{\varepsilon, 12}^2(\bar x)) \prod_{j=3}^nu_{\varepsilon, jj}(\bar x) \\
&=1-\varepsilon^{2}\left[(|x|^{\mu_0}\phi_{\Sigma_0})_{12}(\bar x)\right]^2<1.
\end{align*}
This also leads to contradiction.
\end{itemize}
Hence, \eqref{u-epsilon} holds. Taking $\varepsilon\rightarrow 0$, we have $u\le \frac{|x|^2}{2} $ in $V_\Sigma$. By assumption \eqref{liou-growth},
one gets $u=\frac{|x|^2}{2}$ in $V_\Sigma$.
\end{proof}

\begin{remark}\label{rem-liou-growth}
 If $V_\Sigma$ is a convex cone and $V_\Sigma$ is not the half space, then the growth condition \eqref{liou-growth} can be dropped. Since in this case, from the convexity of the solution and the boundary condition, one has
 \begin{equation*}
 u(x)\le C|x|^2,\quad \forall x\in V_\Sigma.
 \end{equation*}
\end{remark}

\begin{remark}
The condition \eqref{sub-sol-cone} arises naturally if we perform blow-up analysis for \eqref{intro-1} with $C^2$ sub-solution. It can be easily seen that this condition can not be dropped by the following example.  Let $p_0\in V_\Sigma\cap \mathbb S^{n-1}$ be a fixed point. Consider the following Dirichlet problem:
\begin{equation}\label{conic-solution}
\begin{split}
\det D^2 v_{Ra}&=1,\quad \text{in}\quad V_\Sigma\cap B_R(0),\\
v_{Ra}(x)&=\frac 12 |x|^2,\quad \text{on}\quad \partial V_\Sigma\cap B_R(0),\\
v_{Ra}(p_0)&=a\in (-\infty,1/2).
\end{split}
\end{equation}
Then \eqref{conic-solution} admits a solution $v_{Ra}$ by adjusting the boundary data of $v_{Ra}$ on $\partial B_R(0)\cap  V_\Sigma$. Taking $R\rightarrow +\infty$, then $v_{Ra}$ converges to a solution $v_a$ of \eqref{liou-1} which is strictly less than $\frac 12 |x|^2$.
\end{remark}

\begin{remark}\label{emk-liouville1}
Note that $\lambda_1\big(((\mathbb R_{+})^2\times\mathbb R^{n-2})\bigcap\mathbb S^{n-1}\big)=2n$. Let $V_\Sigma\subset \mathbb R^n$ be a convex cone which is bounded by a finite number of planes intersecting at the origin with $\theta(\frac{|x|^2}{2},0,V_\Sigma)\leq \frac{\pi}{2}$. If the number of $(n-1)$ faces of $V_\Sigma$ is greater than or equal to three, or if $V_\Sigma$ has only two $(n-1)$ faces with $\theta(\frac{|x|^2}{2},0,V_\Sigma)<\frac{\pi}{2}$, then
$$\lambda_1\big(V_\Sigma\bigcap\mathbb S^{n-1}\big)>2n.$$
Consequently, Theorem \ref{thm-liou1} is applicable to such $V_\Sigma$.
\end{remark}

For our later use, we present a local estimate in the cone in the following theorem.
\begin{theorem}\label{thm-2+growth}
Let $\Sigma,V_\Sigma$ be given as above with $\lambda_1(\Sigma)=\mu(\mu+n-2)$, $\mu>2$.
Suppose that $ u\in  C(\overline{V_\Sigma\cap B_1(0)})$  is a convex solution of
\begin{equation}\label{cone2+growth}
\begin{split}
\det D^2 u  &=f\quad \text{in } V_\Sigma\cap B_1(0) ,\\
u&=\varphi\quad\text{on }\partial V_\Sigma\cap B_1(0).
\end{split}
\end{equation}
 Assume that
\begin{equation}\label{cone2+growth-condition}
\begin{split}
|f(x)-1|&\leq M |x|^{\beta}\quad \text{in } V_\Sigma\cap B_1(0),\\
\big|\varphi(x)-\frac{|x|^2}{2}\big|&\leq M |x|^{2+\beta}\quad\text{on }\partial V_\Sigma\cap B_1(0),\\
\left|u(x)-\frac{|x|^2}{2}\right| & \le \omega(|x|)|x|^2,\quad \text{in } V_\Sigma\cap B_1(0),
\end{split}
\end{equation}
for some constant $\beta \in (0,1)$, where $M$ is a positive constant and $\omega(t)$ is a non-decreasing function satisfying $\lim_{t\rightarrow 0^+}\omega(t)=0$. Then, there holds
\begin{equation*}
\big|u(x)-\frac{|x|^2}{2}\big|\leq C |x|^{2+\alpha}\quad \text{in } V\cap B_{\delta}(0),
\end{equation*}
where $\alpha\in(0,\beta]$, $\delta\in(0,\frac{1}{10})$ and $C>0$ are constants depending only on
$n$, $\mu$, $M$ and $\omega(t)$.
\end{theorem}
\begin{proof}
Let $\Sigma\subset\subset \Sigma'\subset \mathbb S^{n-1}$ be such that $\lambda_1(\Sigma')=\mu'(n-2+\mu')$ for some $\mu'>2$.
Then,
$$0<c_1 \leq\phi_{\Sigma'} \leq c_2<+\infty,\quad \text{in}\quad  \overline{\Sigma}.$$
Take
\begin{equation*}
\alpha=\min\left(\beta,\frac{\mu'-2}{2}\right),\quad \gamma=\mu'(n-2+\mu')-(2+\alpha)(\alpha+n)>0.
\end{equation*}
Then,
\begin{equation*}
-\Delta\big(|x|^{2+\alpha} \phi_{\Sigma'})=\gamma|x|^{\alpha} \phi_{\Sigma'}, \text{ in } V_{\Sigma}.
\end{equation*}
Let $c_3>0$ be such that
 \begin{equation*}
\sum_{i,j=1}^n\big|\partial_{ij}(|x|^{2+\alpha} \phi_{\Sigma'})\big|\leq c_3 |x|^{\alpha} \phi_{\Sigma'}, \text{ in } V_\Sigma.
\end{equation*}
Set $A=c_1^{-1}\max\{(1+2\gamma^{-1})M,\omega(\delta)\delta^{-\alpha}\}$ where $\delta\in (0,\frac{1}{10})$ is chosen to satisfy
 $$ 10n^{4n}A\delta^{\alpha}c_2 c_3\leq \min\{\frac{\gamma}{2},1\}.$$
Denote
 $$u_+(x)=\frac{|x|^2}{2}+A|x|^{2+\alpha} \phi_{\Sigma'}.$$
  By \eqref{cone2+growth-condition} and the choice of $\delta$ and $A$, it is easy to check  that $u_+$ is convex in $V\cap B_{\delta}(0)$ and
 $$u\leq u_+, \text{ on } \partial \big(V\cap B_{\delta}(0)\big).$$
By a direct computation, we have
\begin{equation*}
\begin{split}
\det D^2 u_+ &\leq 1+\Delta(A|x|^{2+\alpha} \phi_{\Sigma'})+\frac{\gamma}{2}(A|x|^{\alpha} \phi_{\Sigma'})\\
&\leq 1-\frac{\gamma}{2}A|x|^{\alpha}\\
&\leq 1-M|x|^{\beta}\\
&\leq f \quad \text{in } V\cap B_{\delta}(0).
\end{split}
\end{equation*}
Hence, by the maximum principle,  $u\leq u_+$ in $V\cap B_{\delta}(0)$.

Set $u_-=\frac{|x|^2}{2}-A|x|^{2+
\alpha} \phi_{\Sigma'}$.
We can derive the estimate $u_{-}\leq u$ in $V\cap B_{\delta}(0)$ similarly.
\end{proof}

In order to prove Theorem \ref{mainthm1}, we need to pay special attention to the  rigidity of the solutions of equation \eqref{liou-1} in $V =(\mathbb R_{+})^2\times\mathbb R^{n-2} $. For some $\mu\in(0,1)$, set
\begin{equation}\label{def-Amu}
 A'_{\mu}= \begin{pmatrix}
1 & -\cot(\mu \pi)\\[3pt]
 0 & \csc(\mu \pi)
 \end{pmatrix}_{2\times2}
 ,\quad
A_{\mu}= \begin{pmatrix}
A'_{\mu}&O_{2\times(n-2)}\\[3pt]
O_{(n-2)\times2} & I_{(n-2)\times(n-2)}
 \end{pmatrix}_{n\times n}.
\end{equation}

 Then, $$A_{\mu}(V_\mu\times\mathbb R^{n-2})=(\mathbb R_{+})^2\times\mathbb R^{n-2},$$
  \begin{equation*}
\frac{|x|^2}{2}+\cos(\mu\pi) x_1x_2=\left(\frac{|x|^2}{2}\right)\circ A_{\mu}^{-1}\quad \text{in}\quad  (\mathbb R_{+})^2\times\mathbb R^{n-2}.
\end{equation*}

Moreover, if $u$ solves \eqref{liou-1} for $V=V_\mu\times\mathbb R^{n-2}$ if and only if $u\circ A_{\mu}^{-1}$ solves
\begin{equation*}
\begin{split}
\det D^2 u  &=\sin^2(\mu \pi)\quad \text{in } (\mathbb R_{+})^2\times\mathbb R^{n-2},\\
u&=\frac{|x|^2}{2}\quad\text{on }\partial \big((\mathbb R_{+})^2\times\mathbb R^{n-2}\big).
\end{split}
\end{equation*}

\begin{lemma}\label{lemma-liou1}
Let $c\in(0,1]$ and $V= (\mathbb R_{+})^2\times\mathbb R^{n-2}$.
Suppose $u\in  C(\overline{V})$ is a convex solution of
\begin{equation}\label{liou-2}
\begin{split}
\det D^2 u  &=c\quad \text{in } V ,\\
u&=\frac{|x|^2}{2}\quad\text{on }\partial V.
\end{split}
\end{equation}
In addition, if
 $u\ge A|x|^2$ for some positive constant $A\in (0,1)$,  then
\begin{equation*}
u(x)\leq \frac{|x|^2}{2}+\sqrt{1-c}x_1x_2  \quad \text{in } V .
\end{equation*}

\end{lemma}

\begin{proof}
Since $u$ has quadratic growth at infinity, we know $u\in C^\infty(V)$.
\par
{\it Step 1.} There exist two  constants $\varepsilon,\sigma>0$ such that
\begin{equation*}
u(x)\le \left(\frac 12 +\varepsilon\right)|x|^2+(\sqrt{1-c}+\sigma)x_1x_2=P_{\varepsilon,\sigma},\quad \text{in}\quad V.
\end{equation*}
In fact, for any $\varepsilon>0$, there exists a $\theta_\varepsilon>0$ such that
\begin{equation}\label{0321}
u(x)\le \left(\frac 12+\varepsilon\right)|x|^2,\quad \text{in}\quad  \{x|0<\min(x_1,x_2)<|x|\sin\theta_\varepsilon\}.
\end{equation}
Suppose not, then there exists a sequence of points $\{x^m\}_{m=1}^{+\infty}$ and a positive constant $\varepsilon_0$ such that
\begin{equation*}
u(x^m)\ge \left(\frac 12 +\varepsilon_0\right) |x^m|^2, \quad d\left(\frac{x^m}{|x^m|},\partial V\right)\rightarrow 0,\quad m\rightarrow +\infty.
\end{equation*}
Let
\begin{equation*}
u_m(x)=\frac{u(\lambda_m x)}{\lambda_m^2},\quad \lambda_m=|x^m|, \quad y_m=\frac{x^m}{|x^m|}\in \mathbb S^{n-1}\cap V.
\end{equation*}
Then, up to a subsequence, $y_m\rightarrow \bar y\in \partial V\cap \mathbb S^{n-1}$.  Let $\overline H$ be defined as follows:
\begin{equation*}
\overline{H}(x)=\sup\{\ell(x)| \ell(y)\le u(y),\	\forall  y\in \partial V\}=\frac{1}{2}|x|^2+x_1x_2.
\end{equation*}
Let $\underline H$ be defined as follows:
\begin{equation*}
\begin{split}
& \det D^2 \underline H=1,\quad \text{in}\quad V\cap B_{R_0},\\
& \underline H=\psi,\quad \text{on}\quad \partial(V\cap B_{R_0})
\end{split}
\end{equation*}
where
\begin{equation*}
\psi(x)=\sup\{\ell_p(x)|\ell_p(x)=\frac 12 |p|^2+p\cdot (x-p),\	p\in B_2(0)\}
\end{equation*}
and $R_0$ is chosen such that $\psi(x)\le  A|x|^2$ on $\partial B_{R_0}$.  Then, it follows from maximum principle that
$$\underline H\le u_m\le \overline H,\quad \text{in}\quad V\cap B_2(0).$$
Since $\underline H,\overline H,u_m$ have the same boundary value $\frac 12 |x|^2$ on $\partial V\cap B_2(0)$, this implies $y_m\rightarrow \bar y\in \partial V$ is impossible and proves \eqref{0321}.
\par Then, for any $x\in  \{x|\min(x_1,x_2)>|x|\sin\theta_\varepsilon\}$, for $\varepsilon$ small, we can choose $\sigma=1-\sqrt{1-c}-\varepsilon>0$
such that
\begin{equation*}
\begin{split}
 \left(\frac 12 +\varepsilon\right)|x|^2+(\sqrt{1-c}+\sigma)x_1x_2
 \ge  \frac 12 (x_1+x_2)^2\ge u(x).
  \end{split}
\end{equation*}
This ends the proof of Step 1.
\par {\it Step 2.}
For fixed small constant $\varepsilon\in (0,\frac{1-\sqrt{1-c}}{4}]$, let $\widetilde \varepsilon>0$ be the unique solution of
\begin{equation*}
(1+2\varepsilon)^{n-2}[(1+2\varepsilon)^2-(\sqrt{1-c}+\widetilde \varepsilon)^2]=c.
\end{equation*}
Claim:
\begin{equation}\label{clm0409}
u(x)\le \left(\frac 12+\varepsilon\right)|x|^2+(\sqrt{1-c}+\widetilde \varepsilon)x_1x_2,\quad \text{in} \quad V.
\end{equation}
Define
\begin{equation*}
\widetilde \sigma=\inf_{\sigma}\{\sigma>0| u(x)\le P_{\varepsilon,\sigma},\quad \text{in} \quad V
\}
\end{equation*}
where $P_{\varepsilon,\sigma}$ is given by Step 1.
By the definition of $\sigma$ is Step 1, we know
\begin{equation*}
(1+2\varepsilon)^{n-2}[(1+2\varepsilon)^2-(\sqrt{1-c}+\sigma)^2]<c
\end{equation*}
for $\varepsilon>0$ small.  Hence,  Claim  \eqref{clm0409} follows from $\widetilde \varepsilon\ge \widetilde \sigma$.  Suppose not, i.e., $\widetilde \sigma>\widetilde \varepsilon$.  Then, by the definition of $\widetilde \varepsilon$, we know $\det D^2 P_{\varepsilon,\widetilde \sigma}<c$ and there exists a sequence of points $ \{x^m\}_{m=1}^{+\infty}\in V$ such that
\begin{equation*}
u(x^m)> P_{\varepsilon,\widetilde \sigma-\frac 1m}(x^m).
\end{equation*}
Recall the notations $u_m(x), \lambda_m,y_m$ in Step 1. Since $u_m\rightarrow \bar u$ locally uniformly in $\overline V$,   we know
\begin{equation*}
y_m=\frac{x^m}{|x^m|}\rightarrow \bar y\in V\cap \mathbb S^{n-1},\quad \bar u(\bar y)=P_{\varepsilon,\widetilde \sigma}(\bar y).
\end{equation*}
This implies
\begin{equation*}
c=\det D^2 \bar u(\bar y)\le \det D^2 P_{\varepsilon,\widetilde \sigma}<c
\end{equation*}
which is a contradiction.  This ends the proof of Claim.
\par Step 3. Since $\varepsilon>0$ can be choose arbitrarily small, we can take $\varepsilon\rightarrow 0$ to get the conclusion of present lemma.
\end{proof}

As a consequence, we have the following Liouville type result for $V=(\mathbb R_{+})^2\times\mathbb R^{n-2}$.
\begin{theorem}\label{thm-liou2}
Let $c=1$ and $u,V$ be as in Lemma \ref{lemma-liou1}. In addition, $u$ satisfies
\begin{equation*}
 u\geq \frac{|x|^2}{2} \quad\text{in } V.
\end{equation*}
 Then,
\begin{equation*}
u(x)=\frac{|x|^2}{2}.
\end{equation*}
\end{theorem}
The following lemma is also important for global $C^2$ estimate.
\begin{lemma}\label{lemma-liou2}
Let $c\in (0,1)$ and $u,V$ be as in Lemma \ref{lemma-liou1}. In addition,
there exists an $\varepsilon_0>0$ such that
$$ u\geq \frac{|x|^2}{2}-\sqrt{1-c}x_1x_2 +\varepsilon_0 x_1x_2, \quad \text{in}\quad V.$$
 Then,
\begin{equation*}
u(x)= \frac{|x|^2}{2}+\sqrt{1-c}x_1x_2  \quad \text{in } V.
\end{equation*}
\end{lemma}

\begin{proof}

Similar as Step 1 in Lemma \ref{lemma-liou1}, for $\forall \varepsilon>0$ small, we have
\begin{equation*}
u(x)\ge \left(\frac 12-\varepsilon\right)|x|^2-(\sqrt{1-c}-\varepsilon-\varepsilon_0)x_1x_2
\end{equation*}
and
\begin{equation*}
(1-2\varepsilon)^{n-2}((1-2\varepsilon)^2-(\sqrt{1-c}-\varepsilon-\varepsilon_0)^2)>c.
\end{equation*}
Then, repeating the arguments of Step 2 in Lemma \ref{lemma-liou1}, we can get
\begin{equation*}
u(x)\ge \left(\frac 12-\varepsilon\right)|x|^2+(\sqrt{1-c}-\widetilde \varepsilon)x_1x_2
\end{equation*}
where $\widetilde \varepsilon$ is the small root of
\begin{equation*}
(1-2\varepsilon)^{n-2}((1-2\varepsilon)^2-(\sqrt{1-c}-\widetilde \varepsilon)^2)=c.
\end{equation*}
Taking $\varepsilon\rightarrow 0^+$, we finish the proof of Lemma \ref{lemma-liou2}.

\end{proof}

Theorem \ref{thm-liou1} and Lemma \ref{lemma-liou2} is crucial in the proof of Theorem \ref{mainthm1}. However, the existence of $C^2(\overline{\Omega})$ sub-solution $u$ is subtle for polytope domain $\Omega$. In the following,  we give two lemmas which enable us to construct $C^2(\overline\Omega)$ sub-solutions in polytopes for $n=2,3$.
\begin{lemma}\label{lemma-liou3}
Let $c\in (0,1)$ and $u,V$ be as in Lemma \ref{lemma-liou1}. In addition, there holds
\begin{equation}\label{eq-liou3-1}
 u\geq \frac{c^{\frac 1n}}{2}|x|^2, \quad \text{in}\quad V.
 \end{equation}
 Then,
\begin{equation*}
u(x)= \frac{|x|^2}{2}+\sqrt{1-c}x_1x_2  \quad \text{in } V.
\end{equation*}
\end{lemma}
\begin{proof}
Suppose  $u(x)\ge \frac A2 |x|^2$ for some positive constant $A$ close enough to $c^{\frac 1n}$.
\par Step 1.  We claim that
there exist small $\varepsilon,\sigma>0$ such that
\begin{equation}\label{eq-del-1}
u(x)\ge \left(\frac A2+\varepsilon\right)|x|^2-\sigma x_1x_2,\quad \text{in}\quad V.
\end{equation}
Similar as Step 1 in Lemma \ref{lemma-liou1},  $\forall \varepsilon\in (0,\frac 12-\frac A2)=(0,\varepsilon_0)$,
there exists $\theta_{\varepsilon_0}>0$ such that
\begin{equation}\label{0321}
u(x)\ge \left(\frac A2+\varepsilon\right)|x|^2,\quad \text{in}\quad  \{x|0<\min(x_1,x_2)<|x|\sin\theta_{\varepsilon_0}\}.
\end{equation}
 Then,  one has
\begin{equation*}
\varepsilon |x|^2-\sigma x_1x_2\le 0,\quad \text{in}  \quad \{x|\min(x_1,x_2)\ge |x|\sin\theta_{\varepsilon_0}\},
\end{equation*}
provided that
\begin{equation}\label{0412-1}
\varepsilon \le \sigma \sin^2\theta_{\varepsilon_0}
\end{equation}
  This implies \eqref{eq-del-1} holds. We further choose $\varepsilon,\sigma$ small such that
\begin{equation*}
(A+2\varepsilon)^{n-2}((A+2\varepsilon)^2-\sigma^2)>c.
\end{equation*}
The above inequality holds if we choose
\begin{equation}\label{0412-2}
\sigma^2\le A\varepsilon,\quad \varepsilon\ge \frac{c-A^n}{(2n-1)A^{n-1}}.
\end{equation}
Combining \eqref{0412-1} and \eqref{0412-2},  we can just choose
\begin{equation*}
\varepsilon=\min(A\sin^4\theta_{\varepsilon_0},\varepsilon_0)\ge\frac{c-A^n}{(2n-1)A^{n-1}},\quad \sigma=\sqrt{A\varepsilon}.
\end{equation*}
By the assumption, the existence of $\varepsilon,\sigma$ satisfying the above inequality is guaranteed.
\par Step 2.  Following the Step 2 of the proof of Lemma \ref{lemma-liou1}, we finally obtain that
\begin{equation*}
u(x)\ge  \left(A/2+\varepsilon\right)|x|^2+\widetilde \varepsilon x_1x_2,\quad \text{in}\quad V
\end{equation*}
where $\widetilde \varepsilon $ is the positive root of
\begin{equation*}
(A+2\varepsilon)^{n-2}((A+2\varepsilon)^2-\widetilde \varepsilon^2)=c.
\end{equation*}
\par Step 3. Replace the constant $A$ by $A+\varepsilon$ and repeat the above arguments in Step 1-2.
Set
\begin{equation*}
A+2\widehat \varepsilon=\sup_{\varepsilon\in(0,\frac{1-A}{2}]}\{A+2\varepsilon| u(x)\ge \left(A/2+\varepsilon\right)|x|^2+\widetilde \varepsilon x_1x_2\}.
\end{equation*}
Then  $A+2\widehat \varepsilon=1$. Otherwise, if $A+2\widehat \varepsilon<1$, We continue to update the constant $A$ by $A+2\widehat \varepsilon$ and repeat the above arguments in Step 1-2 to get a contradiction. This finish the proof of present lemma.

\end{proof}
\begin{remark}\label{remark-liou-3:label}
Following the proof of Lemma \ref{lemma-liou3}, we know the condition \eqref{eq-liou3-1} can also be replaced by the following slightly weaker condition:
\begin{equation*}
u(x)\ge \frac{c^{\frac 1n}-\sigma}{2}|x|^2
\end{equation*}
for some $\sigma>0$ depending only on $n,c,V$.
\end{remark}
We also have the following Liouville type result which is similar to Theorem \ref{thm-liou1}.
\begin{lemma}\label{lem-liou-perturb}
Let $V\subset \mathbb R^n$ be a convex cone which is bounded by a finite number of planes intersecting at the origin with $\theta(\frac{|x|^2}{2},0,V)<\frac{\pi}{2}$.
Suppose that $u\in  C(\overline{V})$ solves \eqref{liou-1}.
Then, there exists a small constant $\delta>0$ such that
\begin{equation*}
u(x)\ge (1/2-\delta)|x|^2,\quad\text{in}\quad V
\end{equation*}
implies $u=\frac 12|x|^2$.
\end{lemma}
\begin{proof}
After an affine transformation, we assume
\begin{equation}
\begin{split}
&\det D^2 u=c,\quad c\in (0,1),\quad \text{in}\quad V\\
& u=\frac 12|x|^2+\sqrt{1-c}x_1x_2 \quad \text{on}\quad \partial V\\
& u(x)\ge (1/2-\delta)|x|^2+\sqrt{1-c}x_1x_2,\quad \text{in}\quad V
\end{split}
\end{equation}
and $\{x_1=0\}\cap \partial V$, $\{x_2=0\}\cap \partial V$ are two $(n-1)-$face of $V$. And  $\delta$ is small enough  such that $(1-2\delta)^n\ge c$.
\par First, we can choose two small constants $\varepsilon,\sigma$ such that
\begin{equation}\label{032101}
u(x)\ge (1/2-\delta+\varepsilon)|x|^2+(\sqrt{1-c}-\sigma)x_1x_2=P_{\varepsilon,\sigma},\quad x\in V.
\end{equation}
In fact, similar as in Step 1 of Lemma \ref{lemma-liou1}, there exists $\Sigma'\subset\subset \Sigma=V\cap \mathbb S^{n-1}$ such that
\begin{equation*}
u(x)\ge \left(\frac{1}{2}-\frac{\delta}{2}\right)|x|^2+\sqrt{1-c}x_1x_2,\quad \text{in}\quad V\backslash V'
\end{equation*}
where  $V'=t\Sigma',$ $t>0$.  In $V'$, we can just choose
\begin{equation*}
\varepsilon\le \min\{o_{\delta}(1)\sigma,\delta/2\}
\end{equation*}
to get \eqref{032101}.
Here, $o_\delta(1)$ is a constant such that $o_\delta(1)\rightarrow 0$ as $\delta\rightarrow 0$.
\par Second, we also need
\begin{equation*}
\begin{split}
\det D^2 P_{\varepsilon,\sigma}=&(1-2\delta+2\varepsilon)^{n-2}\left[(1-2\delta+2\varepsilon)^2-(\sqrt{1-c}-\sigma)^2\right]\\
\ge &[1-2(n-2)(\delta-\varepsilon)][c-4(\delta-\varepsilon)+1.5\sqrt{1-c}\sigma]\ge c.
\end{split}
\end{equation*}
It suffices to choose that
\begin{equation*}
\sigma\ge \frac{2(n-2)+4}{\sqrt{1-c}}(\delta-\varepsilon).
\end{equation*}
Hence, the following choice of $\varepsilon,\sigma$ are suitable
\begin{equation*}
\varepsilon=o_\delta(1) \theta\delta,\quad \sigma=\theta\frac{2(n-2)+4}{\sqrt{1-c}}\delta,
\end{equation*}
where $\theta>0$ is some small constant depending only on $1-c$.
Let $\widetilde \varepsilon$  be the small root of  $\det D^2P_{\varepsilon,\widetilde \varepsilon}=c$.
The rest of the proof follows from Lemma \ref{lemma-liou3}.
\end{proof}

\section{$C^{2}$ regularity up to $(n-2)$-face}
Consider the following Dirichlet problem:
\begin{equation}\label{C2-cd2-eq}
\begin{split}
&\det D^2 u=f \quad \text{in}\quad \Omega,\\
&u=\varphi \quad \text{on}\quad \partial\Omega\bigcap\{x|x_1x_2=0\}.
\end{split}
\end{equation}
We first give an interesting rigidity lemma.
\begin{lemma}\label{cd2-pC2-1}
Let $\Omega=(0,3)^{2}\times (-3,3)^{n-2}$ and $u\in C^{1,1}(\overline{\Omega})\bigcap C^{2}( \overline{\Omega}\setminus (0,0)\times [-3,3]^{n-2} )$ solves \eqref{C2-cd2-eq} with $f,\varphi$ satisfying \eqref{cond-f-varphi}. Suppose there exists a globally $C^2$, convex, sub-solution $\underline u\in C^2(\overline{\Omega})$ to \eqref{C2-cd2-eq}$($i.e. $\det D^2\underline u\ge f$ in $\Omega$; $\underline u=\varphi$ on $\partial\Omega\cap \{x_1x_2=0\}$; $u\ge \underline u$ on $\partial\Omega\cap \{x_1x_2>0\})$.
Assume
\begin{equation}\label{subu-expansion-1}
\underline{u}(x)= \frac{|x|^2}{2}+\underline u_{12}(0) x_1x_2+o(|x|^2),\quad \text{near}\quad 0.
\end{equation}
Then, either $u=\underline{u}$ or there holds
\begin{equation}\label{u-expansion-1}
u= \frac{|x|^2}{2}+\sqrt{1-f(0)}x_1x_2 +o(|x|^2),\quad \text{near}\quad 0.
\end{equation}
\end{lemma}

\begin{proof}
Suppose $u\ne \underline u$.
By strong maximum principle, one obtains
\begin{equation*}
u>\underline{u},\quad \text{in}\quad  (0,3)^2\times(-3,3)^{n-2}.
\end{equation*}
Denote
\begin{equation*}
u_\lambda(x)=\frac{u(\lambda x)}{\lambda^2},\quad \lambda\in (0,1/2).
\end{equation*}
It suffices to show that
\begin{equation*}
u_\lambda\rightarrow \frac{|x|^2}{2}+\sqrt{1-f(0)}x_1x_2,\text{ locally uniformly in } [0,+\infty)^2\times (-\infty,+\infty)^{n-2}
\end{equation*}
as $\lambda\rightarrow 0^+$.  Up to a subsequence, one assumes
\begin{equation*}
u_{\lambda_l}\rightarrow  v \text{ locally uniformly in } \overline{V}=\overline{\mathbb (\mathbb R_{+})^2}\times\mathbb R^{n-2}.
\end{equation*}
 Then, by [Theorem 1.1,\cite{Savin2013}], [Theorem 2,\cite{Caffarelli1990-2}] and Lemma \ref{lem-C11}, $v\in  C^{1,1}(\overline{V})\cap C^{2,\beta}( \overline{V}\setminus (0,0)\times \mathbb R^{n-2} )$ is a convex solution of
\begin{equation*}
\begin{split}
\det D^2 v  &=f(0)\quad \text{in } V ,\\
v&=\frac{|x|^2}{2}\quad\text{on }\partial V,\\
 v&\geq \frac{|x|^2}{2}+\underline{u}_{12}(0)x_1x_2 \quad\text{in } V.
\end{split}
\end{equation*}
Claim: $v=\frac{|x|^2}{2}+\sqrt{1-f(0)}x_1x_2$.
\par
{\it Case 1.} $f(0)=1$. The claim follows from Theorem \ref{thm-liou2}.

{\it Case 2.} $f(0)\in (0,1)$, $\underline{u}_{12}(0)\in (-\sqrt{1-f(0)},\sqrt{1-f(0)}]$. The claim follows from Lemma \ref{lemma-liou2}.

{\it Case 3.} $f(0)\in (0,1),$ $\underline{u}_{12}(0)=-\sqrt{1-f(0) }$. This is the most interesting case.
\par  Denote $\underline{P} =\frac{|x|^2}{2}-\sqrt{1-f(0)}x_1x_2$ and
\begin{equation*}
\underline{u}_\lambda(x)=\frac{\underline{u}(\lambda x)}{\lambda^2},\quad \lambda\in (0,1/2).
\end{equation*}
Since $\underline{u}$ is $C^2$, for small $\epsilon_1>0$, there exists $\Lambda(\epsilon_1)>0$ such that
\begin{equation}\label{u-P}
\|\underline{u}_\lambda-\underline P\|_{C^2([0,2]^2\times [-2,2]^{n-2} )}<\epsilon_1,\quad \forall 0<\lambda<\Lambda(\epsilon_1).
\end{equation}
By Hopf's lemma, for small $\delta>0$, there exists $\epsilon=\epsilon(\delta,\lambda)>0$ such that
\begin{equation*}
u_\lambda \geq\underline{u}_\lambda+\epsilon x_1x_2 \quad\text{on }\partial ([0,\delta]^2)\times [-2,2]^{n-2}.
\end{equation*}
$\forall \widetilde{x}''\in [-1,1]^{n-2}$, define
\begin{equation*}
\overline{u}_\lambda(x)=\underline{u}_\lambda+\epsilon x_1x_2 -|x''-\widetilde{x}''|^2\delta^2 \epsilon.
\end{equation*}
By our choice of $\epsilon$ and $\delta$, there holds
$$\overline{u}_\lambda(x)\leq u_\lambda(x),\quad\text{on}\quad \partial \big( [0,\delta]^2\times [-2,2]^{n-2}\big).$$
A direct calculation yields that
\begin{equation*}
\det  D^2 \overline{u}_\lambda \geq \det  D^2 \underline{u}_\lambda +2\sqrt{1-f(0)}\epsilon-C(n)(\delta^2+\epsilon+\epsilon_1)\epsilon,\quad \text{in} \quad [0,\delta]^2\times [-2,2]^{n-2},
\end{equation*}
where $C(n)$ is some constant depending only on $n$. Hence, for $\delta,\epsilon_1,\epsilon,\lambda$ small enough, we have
\begin{equation*}
\det  D^2 \overline{u}_\lambda \geq \det  D^2 \underline{u}_\lambda\geq \det  D^2 u_\lambda,\quad \text{in} \quad [0,\delta]^2\times [-2,2]^{n-2}.
\end{equation*}
By standard maximum principle, one gets
$$u_\lambda \geq \overline{u}_\lambda, \quad \text{in} \quad [0,\delta]^2\times [-2,2]^{n-2}.$$
In particular, by the arbitrariness of $\widetilde{x}''\in [-1,1]^{n-2}$, for fixed small $\lambda_0$, we have
\begin{equation}\label{u-lambda-esti}
u_{\lambda_0} \geq\underline{u}_{\lambda_0}+\epsilon x_1x_2, \text{ in } [0,\delta]^2\times [-1,1]^{n-2}.
\end{equation}
This implies
\begin{equation}\label{u-lambda-esti-2}
u_\lambda(x) \geq\underline{u}_\lambda(x)+\epsilon x_1x_2, \quad\text{in }\left[0,\frac{\delta\lambda_0}{\lambda}\right]^2\times \left[-\frac{\lambda_0}{\lambda},\frac{\lambda_0}{\lambda}\right]^{n-2},\quad  \forall \lambda\in (0,\lambda_0].
\end{equation}
Taking $\lambda\rightarrow 0$, one obtains
$$v\geq \underline P+\epsilon x_1x_2,\text{ in } (\mathbb R_{+})^2\times\mathbb R^{n-2}.$$
Then, Lemma \ref{lemma-liou2} implies $v=\frac{|x|^2}{2}+\sqrt{1-f(0)}x_1x_2$ .
\end{proof}
\begin{remark}\label{rem-face-2-a}
For the general case, by an affine transformation $T$, we can assume
\begin{equation*}
 (\underline u\circ T^{-1})_{ij}(0)=\delta_{ij}.
\end{equation*}
Then, by an orthogonal transformation $O$, one has
\begin{equation*}
OT((\mathbb R_+)^2\times \mathbb R^{n-2})=V_{\Theta(\underline u,0,\Omega)/\pi}\times \mathbb R^{n-2}.
\end{equation*}
Let
\begin{equation*}
\underline{U}(x)=\underline u\circ T^{-1}\circ O^{-1}\circ A^{-1}_{\Theta(\underline u,0,\Omega)/\pi}(x).
\end{equation*}
Then, $\underline{U}(x)$ satisfies \eqref{subu-expansion-1} in $B_r(0)\cap ((\mathbb R_+)^2\times \mathbb R^{n-2})$ for some small $r>0$.
This implies that we can reduce the general case to the situation in Lemma \ref{cd2-pC2-1}.
\end{remark}

\begin{remark}\label{rem-face-2-b}
Lemma \ref{cd2-pC2-1} implies $u$ is pointwise $C^2$ along $(0,0)\times (-3,3)^{n-2}$. From boundary data $\varphi$, one also knows
\begin{equation*}
u_{ij}(0,0,x'')=\varphi_{ij}(0,0,x''),\quad i,j=1,\cdots,n,\quad i+j\ne 3.
\end{equation*}
By expansion of \eqref{C2-cd2-eq}, then either $u=\underline u$ or
$u_{12}(0,0,x'')$ is the big root of the following univariate quadratic equation
\begin{equation*}
a u_{12}^2(0,0,x'')+b u_{12}(0,0,x'')+c=0
\end{equation*}
which is equivalent to the equation
\begin{equation*}
\det D^2 u=f,\quad \text{at}\quad (0,0,x'').
\end{equation*}
 Hence, $u_{12}(0,0,x'')$  is continuous along $(0,0)\times (-3,3)^{n-2}$.
\end{remark}

\begin{remark} Lemma \ref{cd2-pC2-1} also implies the regularity at $\Gamma_{n-2}(\Omega)\backslash \Gamma_{n-3}(\Omega)$ can't be determined locally.  Suppose $u_\varepsilon$ be an Alexandrov solution of
\begin{equation*}
\begin{split}
&\det D^2 u_\varepsilon=3/4,\quad \text{in}\quad (0,1)^2,\\
&u_\varepsilon=\frac{1}{2}|x|^2-(\frac 12+\varepsilon )x_1x_2,\quad \text{on}\quad \partial([0,1]^2),
\end{split}
\end{equation*}
where $\varepsilon\in(-\frac{1}{4},\frac{1}{4})$.
Note that the boundary value of $u_\varepsilon$ near $(0,0)$ is independent of $\varepsilon$. However, from Lemma \ref{cd2-pC2-1}, $u_\varepsilon$ can't be $C^2$ at $(0,0)$ for $\varepsilon>0$ and is $C^2$ at $(0,0)$ for $\varepsilon\le 0$.
\end{remark}

\begin{theorem}\label{C2-cd2}
Suppose all the assumptions except \eqref{subu-expansion-1} in Lemma \ref{cd2-pC2-1} are fulfilled. Then,
$u$ is $C^2$ up to $(0,0)\times [-1,1]^{n-2}$.
\end{theorem}
\begin{proof}
It is enough to prove that $u$ is $C^2$ up to the origin. Without loss of generality, we may assume the assumptions of Lemma \ref{cd2-pC2-1} are satisfied and $u>  \underline{u}$. Hence, $u\in C^2\{x^0\}$, $\forall x^0\in (0,0)\times [-1,1]^{n-2}$ and
\begin{equation*}
u_{ij}(0)=\delta_{ij}+\sqrt{1-f(0)}(\delta_{i1}\delta_{j2}+\delta_{i2}\delta_{j1}).
\end{equation*}
Consider a sequence of points $ \Omega \ni x^{(l)}\rightarrow 0$ and the following cases.
\begin{itemize}
	\item[(1).] $\lim_{l\rightarrow +\infty}\frac{x^{(l)}}{|x^{(l)}|}=\bar x\in \mathbb S^{n-1}\cap \{x_1^2+x_2^2>0\}$. Define
\begin{equation*}
u^{(l)}(x)=\frac{u(\lambda_l x)}{\lambda_l^2},\quad \lambda_l=|x^{(l)}|.
\end{equation*}
By Lemma \ref{cd2-pC2-1}, there holds
	\begin{equation*}
	u^{(l)}(x)\rightarrow \frac{1}{2}|x|^2+\sqrt{1-f(0)}x_1x_2,\text{ locally uniformly in } \overline{(\mathbb R_+)^2}\times\mathbb R^{n-2}.
	\end{equation*}
	Then, boundary estimates(Theorem 1.1,\cite{Savin2013}) and interior estimate(Theorem 2,\cite{Caffarelli1990-2}) also imply the above convergence is $C^2$ locally in $(\overline{(\mathbb R_+)^2}\times\mathbb R^{n-2})\backslash\{x_1=x_2=0\}$. Then, one obtains
	\begin{equation*}
	u_{ij}(x^{(l)})=u^{(l)}_{ij}\left(\frac{x^{(l)}}{|x^{(l)}|}\right)\rightarrow u_{ij}(0).
	\end{equation*}
	\item[(2).] $\lim_{l\rightarrow +\infty}\frac{x^{(l)}}{|x^{(l)}|}=\bar x\in \mathbb S^{n-1}\cap \{x_1=x_2=0\}$. Let $\ell^{(l)}(x)$ be the supporting plane of $u,\underline u$ at $y^{(l)}=(0,0,\widetilde{x^{(l)}})$ where $\widetilde{x^{(l)}}=(x^{(l)}_3,\cdots,x^{(l)}_n)$. Define
	\begin{equation*}
	u^{(l)}(x)=\frac{(u-\ell^{(l)})(\lambda_l x+y^{(l)})}{\lambda_l^2}, \underline{u}^{(l)}(x)=\frac{(\underline u-\ell^{(l)})(\lambda_l x+y^{(l)})}{\lambda_l^2},\lambda_l=\sqrt{(x^{(l)}_1)^2+(x^{(l)}_2)^2}.
	\end{equation*}
	Then $z^{(l)}=\left(\frac{x^{(l)}_1}{\lambda_l},\frac{x^{(l)}_2}{\lambda_l},0,\cdots,0\right)\rightarrow \bar z\in \mathbb S^1 \times (0,\cdots,0)$. It should be emphasized that we can't use Lemma \ref{cd2-pC2-1} directly in this case.
	\begin{itemize}
		\item[(2.1).] $f(0)=1$. Boundary estimate(Theorem 1.1,\cite{Savin2013}) and interior estimate(Theorem 2,\cite{Caffarelli1990-2}) imply that
		\begin{equation*}
		u^{(l)}(x)\rightarrow P\text{ locally uniformly in } C^2((\overline{(\mathbb R_+)^2}\times\mathbb R^{n-2})\backslash\{x_1=x_2=0\})
		\end{equation*}
		 where $P$ is a solution of \eqref{liou-2} with $c=1$ and $P\ge \frac 12|x|^2$.  Theorem \ref{thm-liou2} implies $P=\frac 12 |x|^2$. Hence, there holds
		\begin{equation*}
		u_{ij}(x^{(l)})=u_{ij}^{(l)}(z^{(l)})\rightarrow \delta_{ij},\quad l\rightarrow +\infty.
		\end{equation*}
	\item[(2.2).] $f(0)\in(0,1)$ and $\underline u_{12}(0)\in (-\sqrt{1-f(0)},\sqrt{1-f(0)}]$. Similar to (2.1),
	\begin{equation*}u^{(l)}(x),\underline{u}^{(l)}(x)\rightarrow P,\underline{P}
	\end{equation*}
	 where $P$ is a solution of \eqref{liou-2} with $c=f(0)\in (0,1)$ and
	\begin{equation*}
	P(x)\ge \underline{P}(x)=\frac 12|x|^2+\underline{u}_{12}(0)x_1x_2,\text{ in }(\mathbb R_+)^2\times\mathbb R^{n-2}.
	\end{equation*}
By Lemma \ref{lemma-liou2}, $P=\frac 12 |x|^2+\sqrt{1-f(0)}x_1x_2$ and this implies
	\begin{equation*}
		u_{ij}(x^{(l)})=u_{ij}^{(l)}(z^{(l)})\rightarrow \delta_{ij}+\sqrt{1-f(0)}(\delta_{i1}\delta_{j2}+\delta_{i2}\delta_{j1}),\quad l\rightarrow +\infty.
		\end{equation*}
	\item[(2.3).] $f(0)\in(0,1)$ and $\underline u_{12}(0)=-\sqrt{1-f(0)}$.
	From the proof of Lemma \ref{cd2-pC2-1}, there exist $\epsilon,\delta,\lambda_0>0$ such that
	\begin{equation}
	\frac{u(\lambda x)-\underline u(\lambda x)}{\lambda^2}\ge \epsilon x_1x_2,\quad x\in [0,\delta]^{2}\times [-1,1]^{n-2},\quad \lambda\in (0,\lambda_0).
	\end{equation}
	Notice that $\frac{\lambda_l}{|y^{(l)}|}\rightarrow 0$ as $l\rightarrow +\infty$. Hence, $\forall R>0$,  for large $l$, there holds
	\begin{equation*}
	\frac{\lambda_l x+y^{(l)}}{|\lambda_l x+y^{(l)}|}\in [0,\delta]^{2}\times [-1,1]^{n-2},\quad \forall x\in B_R(0)\cap ((\mathbb R_+)^2\times\mathbb R^{n-2}).
	\end{equation*}
	Then
	\begin{equation*}
	\begin{split}
	u^{(l)}(x)-\underline u^{(l)}(x)=&\frac{u(\lambda_l x+y^{(l)})-\underline u(\lambda_l x+y^{(l)})}{|\lambda_l x+y^{(l)}|^2}\cdot \frac{|\lambda_l x+y^{(l)}|^2}{\lambda_l^2}\\
	\ge &\epsilon x_1x_2,\quad \text{in}\quad B_R(0)\cap ((\mathbb R_+)^2\times\mathbb R^{n-2}).
	\end{split}
	\end{equation*}
	Similar to (2.1), $u^{(l)}(x)\rightarrow P$ where $P$ is a solution of \eqref{liou-2} with $c=f(0)\in (0,1)$ and
	\begin{equation*}
	P(x)\ge \frac 12 |x|^2-\sqrt{1-f(0)}x_1x_2+\epsilon x_1x_2,\text{ in }(\mathbb R_+)^2\times\mathbb R^{n-2}.
	\end{equation*}
	By Lemma \ref{lemma-liou2}, $P=\frac 12 |x|^2+\sqrt{1-f(0)}x_1x_2$ and this implies
	\begin{equation*}
		u_{ij}(x^{(l)})=u_{ij}^{(l)}(z^{(l)})\rightarrow \delta_{ij}+\sqrt{1-f(0)}(\delta_{i1}\delta_{j2}+\delta_{i2}\delta_{j1}),\quad l\rightarrow +\infty.
		\end{equation*}
	\end{itemize}
\end{itemize}
This ends the proof of Theorem \ref{C2-cd2}.
\end{proof}

From Lemma \ref{lemma-liou2}, and the blow up argument used in the proof of Theorem \ref{C2-cd2}, we have the following proposition.
\begin{prop}\label{ modulusofcontinuity1}
Suppose all the assumptions in Theorem \ref{C2-cd2} are satisfied. In addition, there exists an $\epsilon>0$ such that
\begin{equation*}
\det D^2 \underline{u}-f\geq \epsilon>0,\text{ on } (0,0)\times [-1,1]^{n-2}.
\end{equation*}
 Then, the modulus of continuity of $D^2 u$ along $ (0,0)\times [-1,1]^{n-2}$ depends only on $n$, $\epsilon$, $\beta$, $\|f\|_{C^\beta(\overline{\Omega})}$, $\|\frac{1}{f}\|_{L^{\infty}(\Omega)}$, $\|\underline{u} \|_{C^{2} (\overline{\Omega})}$, $\|\varphi\|_{C^{2,\beta}(\overline{\Omega})}$,
and the modulus of continuity of $D^2 \underline{u}$ in $\overline{\Omega}$.
\end{prop}
The condition of Proposition \ref{ modulusofcontinuity1} can also be replaced by the following geometric condition.
\begin{prop}\label{ modulusofcontinuity2}
Suppose all the assumptions in Theorem \ref{C2-cd2} are satisfied. In addition, there exists $\epsilon>0$ such that
\begin{equation*}
\frac{\pi}{2}-\Theta(\underline{u},x,\Omega)\geq \epsilon,\quad \forall x\in (0,0)\times [-1,1]^{n-2}.
\end{equation*}
Then, the modulus of continuity of $D^2 u$ along $ (0,0)\times [-1, 1]^{n-2}$ depends only on $n$, $\epsilon$, $\beta$, $\|f\|_{C^\beta(\overline{\Omega})}$, $\|\frac{1}{f}\|_{L^{\infty}(\Omega)}$,  $\|\underline{u} \|_{C^{2} (\overline{\Omega})}$, $\|\varphi\|_{C^{2,\beta}(\overline{\Omega})}$,
and the modulus of continuity of $D^2 \underline{u}$ in $\overline{\Omega}$.
\end{prop}

\section{$C^{2}$ regularity up to $(n-k)$-face, $3\leq k\leq n$}
Consider the following Dirichlet problem:
\begin{equation}\label{C2-cdk-eq}
\begin{split}
&\det D^2 u=f \quad \text{in}\quad \Omega=(\mathbb V_k\times (-3,3)^{n-k})\cap B_3(0),\\
&u=\varphi \quad \text{on}\quad \partial\Omega\cap  (\partial\mathbb V_k\times [-3,3]^{n-k}).
\end{split}
\end{equation}

\begin{theorem}\label{thm-c2-cdk}
Let $u\in C^{1,1}(\overline{\Omega})\bigcap C^{2}(\overline{\Omega}\setminus (\partial \mathbb V_k\times [-3,3]^{n-k}))$ solves \eqref{C2-cdk-eq} with $0<f\in C^\beta(\overline{\Omega}),\varphi\in C^{2,\beta}(\overline{\Omega})$. In addition, there exists a globally $C^2$, convex, sub-solution $\underline u\in C^2(\overline{\Omega})$ to \eqref{C2-cdk-eq}(i.e. $\det D^2\underline u\ge f$ in $\Omega$; $\underline u=\varphi$ on $\partial \mathbb V_k\times [-3,3]^{n-k})$; $u\ge \underline u$ on $\mathbb V_k\times \partial([-3,3]^{n-k})$).
Suppose $u>\underline u$ in $\Omega$. Then, $u\in C^2\big(\overline{\Omega} \bigcap B_1(0) \big)$ if and only if
\begin{equation}\label{C2-ns-c}
\begin{split}
&\det D^2 \varphi=f,\text{ in } \Gamma_{n-3}(\Omega)\cap B_1(0),\\
& \Theta(\varphi,x,\Omega)\le \frac{\pi}2,\quad \forall x\in \Gamma_{n-3}(\Omega)\cap B_1(0).
\end{split}
\end{equation}
Moreover, the modulus of continuity of $D^2 u$ at $0$ depends only on $n$, $k$, $\beta$, $\|\varphi\|_{C^{2,\beta}(\overline{\Omega})}$,
$\|\frac{1}{f}\|_{L^{\infty}(\Omega)}$, $\|f\|_{C^\beta(\overline{\Omega})}$, $\|\underline{u} \|_{C^{2} (\overline{\Omega})}$,
and the modulus of continuity of $D^2 \underline{u}$ in $\overline{\Omega}$.
\end{theorem}

\begin{proof}
Step 1. \eqref{C2-ns-c} is necessary.

The first part of \eqref{C2-ns-c} is a compatible condition. We only need to consider the second part of \eqref{C2-ns-c}. Suppose this part is violated at $x=0$.  Up to an affine transformation,  we assume $D^2 \varphi(0)=I_{n\times n}$. Suppose that there exist two adjacent $(n-1)$ faces with dihedral angle $\theta=\mu\pi$, $\mu>\frac 12$, i.e., a subset of $V_{\mu}\times \mathbb R^{n-2}$ with $\mu>\frac{1}{2}$.
Suppose  $u$ is $C^2$ near $0$, then
 \begin{equation*}
 (u \circ A_{\mu}^{-1})_{12}(0)=(\underline{u} \circ A_{\mu}^{-1})_{12}(0)<0,
 \end{equation*}
 where $A_{\mu}$ is given by \eqref{def-Amu}.
 However, by Lemma \ref{cd2-pC2-1} and Remark \ref{rem-face-2-a}-\ref{rem-face-2-b},  there holds
 $$(u \circ A_{\mu}^{-1})_{12}(0)=\lim_{x''\rightarrow 0} (u \circ A_{\mu}^{-1})_{12}(0',x'')\ge 0.$$
 This yields a contradiction.

Step 2. \eqref{C2-ns-c} is sufficient.

It suffices to prove $u$ is $C^2$ up to $x=0$. By taking an affine transformation,  we can assume
\begin{equation*}
D^2\varphi(0)=I_{n\times n},\	\varphi(0)=|\nabla \varphi|=0,\quad \Theta(\varphi,0,\mathbb V_k)\le \frac{\pi}2.
\end{equation*}
Consider a sequence of points $\{x^{(m)}\}\in \big(\mathbb V_k \times [-3,3]^{n-k}\big)\bigcap B_{\frac{1}{4}}(0)$.
Claim:
\begin{equation*}
D^2 u(x^{(m)})\rightarrow I_{n\times n},\quad x^{(m)}\rightarrow 0.
\end{equation*}
Denote $x=(x',x'')$ where $x'=(x_1,\cdots,x_k)$, $x''=(x_{k+1},\cdots,x_n)$.
Consider
\begin{equation*}
u^{(m)}(x)=\frac{(u-\ell^{(m)})(\lambda_m x',\lambda_m x''+(x^{(m)})'')}{\lambda_m^2},\quad \lambda_m=|(x^{(m)})'|,
\end{equation*}
where $\ell^{(m)}(x)$ is the supporting plane of $u$ at $(0',(x^{(m)})'')$.
 Up to a subsequence, one assumes
\begin{equation*}
y^{(m)}=\frac{(x^{(m)})'}{\lambda_m}\rightarrow \bar y\in \mathbb S^{k-1}\cap \overline{\mathbb V_k}.
\end{equation*}
\begin{itemize}
	\item[(1).] $\bar y\in (\overline {\mathbb V_k}\backslash \Gamma_{k-2}(\mathbb V_k))\cap \mathbb S^{k-1}$.
	Boundary estimate(Theorem 1.1,\cite{Savin2013}) and interior estimate(Theorem 2,\cite{Caffarelli1990-2}) imply that
	$$u^{(m)}\rightarrow v,\text{ locally uniformly in } C^2_{loc}((\overline {\mathbb V_k}\backslash \Gamma_{k-2}(\mathbb V_k))\times \mathbb R^{n-k}),$$
where $v$ solves
\begin{equation*}
\begin{split}
\det D^2 v  &=1\quad \text{in } \mathbb V_{k}\times \mathbb R ^{n-k} ,\\
v&=\frac{|x|^2}{2}\quad\text{on }\partial ( \mathbb V_{k}\times \mathbb R^{n-k}).\\
v&\geq \frac{|x|^2}{2}\quad\text{in } \mathbb V_{k}\times \mathbb R^{n-k}.
\end{split}
\end{equation*}
Hence, by Theorem \ref{thm-liou1}, Remark \ref{rem-liou-growth} and Remark \ref{emk-liouville1}, we obtain $v=\frac{|x|^2}{2}$. This implies
\begin{equation*}
u_{ij}(x^{(m)})=u^{(m)}_{ij}\left(y^{(m)},0''\right)\rightarrow \delta_{ij},\text{ as }m\rightarrow +\infty.
\end{equation*}
\item[(2).] $\bar y\in (\Gamma_{k-2}(\mathbb V_k)\backslash \Gamma_{k-3}(\mathbb V_k))\cap \mathbb S^{k-1}$. Without loss of generality, we assume $\bar y_3,\cdots,\bar y_{k}>0$, $\bar y_1=\bar y_2=0$. Denote
\begin{equation*}
\bar y^{(m)}=(0,0,y^{(m)}_3,\cdots,y^{(m)}_k,0''),\quad \tilde \lambda_m=\sqrt{(y^{(m)}_1)^2+(y^{(m)}_2)^2}.
\end{equation*}
Up to a subsequence, one assumes
\begin{equation*}
\bar z^{(m)}=\left(\frac{y^{(m)}_1}{\tilde \lambda_m},\frac{y^{(m)}_2}{\tilde \lambda_m},0,\cdots,0\right)\rightarrow \bar z\in \mathbb S^1\times \mathbb R^{n-2}.
\end{equation*}
Define
\begin{equation*}
\tilde u^{(m)}(x)=\frac{(u^{(m)}-\tilde \ell^{(m)})(\tilde \lambda_m x+\bar y^{(m)})}{\tilde \lambda_m^2}
\end{equation*}
where $\tilde \ell^{(m)}$ is the supporting plane of $u^{(m)}$ at $\bar y^{(m)}$. Boundary estimate(Theorem 1.1,\cite{Savin2013}) and interior estimate(Theorem 2,\cite{Caffarelli1990-2}) imply that
	$$\tilde u^{(m)}\rightarrow v,\text{ locally uniformly in } C^2_{loc}((\overline { V_\mu}\backslash   \Gamma_{0} (V_\mu))\times \mathbb R^{n-2}),\quad \mu\in (0,1/2)$$
where $v$ solves
\begin{equation*}
\begin{split}
\det D^2 v  &=1\quad \text{in } V_{\mu}\times \mathbb R ^{n-2} ,\\
v&=\frac{|x|^2}{2}\quad\text{on }\partial ( V_{\mu}\times \mathbb R^{n-2}),\\
v&\geq \frac{|x|^2}{2}\quad\text{in } V_{\mu}\times \mathbb R^{n-2}.
\end{split}
\end{equation*}
Hence, by Theorem \ref{thm-liou1}, Remark \ref{rem-liou-growth} and Remark \ref{emk-liouville1}, we obtain $v=\frac{|x|^2}{2}$. This implies
\begin{equation*}
u_{ij}(x^{(m)})=u^{(m)}_{ij}\left(y^{(m)},0''\right)=\tilde u^{(m)}_{ij}(\bar z^{(m)})\rightarrow \delta_{ij},\text{ as }m\rightarrow +\infty.
\end{equation*}
\item[(3).] $\bar y\in (\Gamma_{l}(\mathbb V_k)\backslash \Gamma_{l-1}(\mathbb V_k))\cap \mathbb S^{k-1}$, $1\le l\le k-3$, $k\ge 4$. We may assume
\begin{equation*}
\bar y_{k-l+1},\cdots \bar y_{k}>0,\quad \bar y_{1}=\cdots=\bar y_{k-l}=0
\end{equation*}
Denote
\begin{equation*}
\bar y^{(m)}=(0,\cdots,0,y^{(m)}_{k-l+1},\cdots,y^{(m)}_k,0''),\quad \tilde \lambda_m=\sqrt{(y^{(m)}_1)^2+\cdots+(y^{(m)}_{k-l})^2}.
\end{equation*}
Up to a subsequence, one assumes
\begin{equation*}
\bar z^{(m)}=\left(\frac{y^{(m)}_1}{\tilde \lambda_m},\cdots,\frac{y^{(m)}_{k-l}}{\tilde \lambda_m},0,\cdots,0\right)\rightarrow \bar z\in \mathbb S^{k-l-1}\times \mathbb R^{n-(k-l)}.
\end{equation*}
Define
\begin{equation*}
\tilde u^{(m)}(x)=\frac{(u^{(m)}-\tilde \ell^{(m)})(\tilde \lambda_m x+\bar y^{(m)})}{\tilde \lambda_m^2}
\end{equation*}
where $\tilde \ell^{(m)}$ is the supporting plane of $u^{(m)}$ at $\bar y^{(m)}$.
$$\tilde u^{(m)}\rightarrow v,\text{ locally uniformly in } C^2_{loc}((\overline {\mathbb  V_{k-l}}\backslash \Gamma_{k-l-2}(\mathbb V_{k-l}))\times \mathbb R^{n-(k-l)}),$$
where $v$ solves
\begin{equation*}
\begin{split}
\det D^2 v  &=1\quad \text{in } \mathbb V_{k-l}\times \mathbb R ^{n-(k-l)} ,\\
v&=\frac{|x|^2}{2}\quad\text{on }\partial ( \mathbb V_{k-l}\times \mathbb R^{n-(k-l)}).\\
v&\geq \frac{|x|^2}{2}\quad\text{in } \mathbb V_{k-l}\times \mathbb R^{n-(k-l)}.
\end{split}
\end{equation*}
This reduces the situation from $\mathbb V_k\times \mathbb R^{n-k}$ to $\mathbb V_{k-l}\times \mathbb R^{n-(k-l)}$, $l\ge 1$. We can repeat the above arguments (1)-(3) to get the desired results.
\end{itemize}
This completes the proof of present theorem.
\end{proof}

It is evident that the property of the existence of a global $C^2$ function satisfying the $A-$condition itself imposes constraints on the geometry of the polytope $P$. To end this section, we prove that
the existence of a convex function $u\in C^2(\overline P)$ satisfying the $A-$ condition implies that the convex polytope $P$ is simple.

\begin{prop}\label{prop5.1:label}
Let $\Omega$ be a convex polytope in $\mathbb R^n$. If  there exists a convex function $u\in C^2(\overline{\Omega})$ satisfying the $A-$condition, then $\Omega$ must be simple. The converse is also true for $n=2,3$.
\end{prop}

\begin{proof}
Let $u\in C^2(\overline{\Omega})$ be such a convex function. Suppose $\Omega$ is not simple. Then, there exists a vertex of $\Omega$ such that it belongs to $F_i^{(0)}$, $i=1,\cdots,N$ where $F_i^{(0)}$ is an $(n-1)-$face of $\Omega$ and $N\ge n+1$.
Without loss of generality, we may assume this vertex is the origin $O$. By a rotation of the coordinates, we can make
\begin{equation*}
\Omega\subset\{x_n>0\},\quad \partial\Omega \cap \{x_n=0\}=O.
\end{equation*}
Denote the interior unit normal of $F_i^{(0)}$ by $\vec \nu_i^{(0)}$. Let
\begin{equation*}
\Omega^{(1)}=\Omega\cap \{x_n=\delta_1\},\quad F_i^{(1)}=F_i^{(0)}\cap \{x_n=\delta_1\}
\end{equation*}
for some $\delta_1>0$ small.  Denote the  the interior unit normal of $F_i^{(1)}$ relative to $\Omega^{(1)}$ by  $\vec \nu_i^{(1)}$(See Figure \ref{fig1}).
\begin{figure}[ht]
\begin{minipage}[t]{0.45\linewidth}
\centering
\includegraphics[width=7 cm]{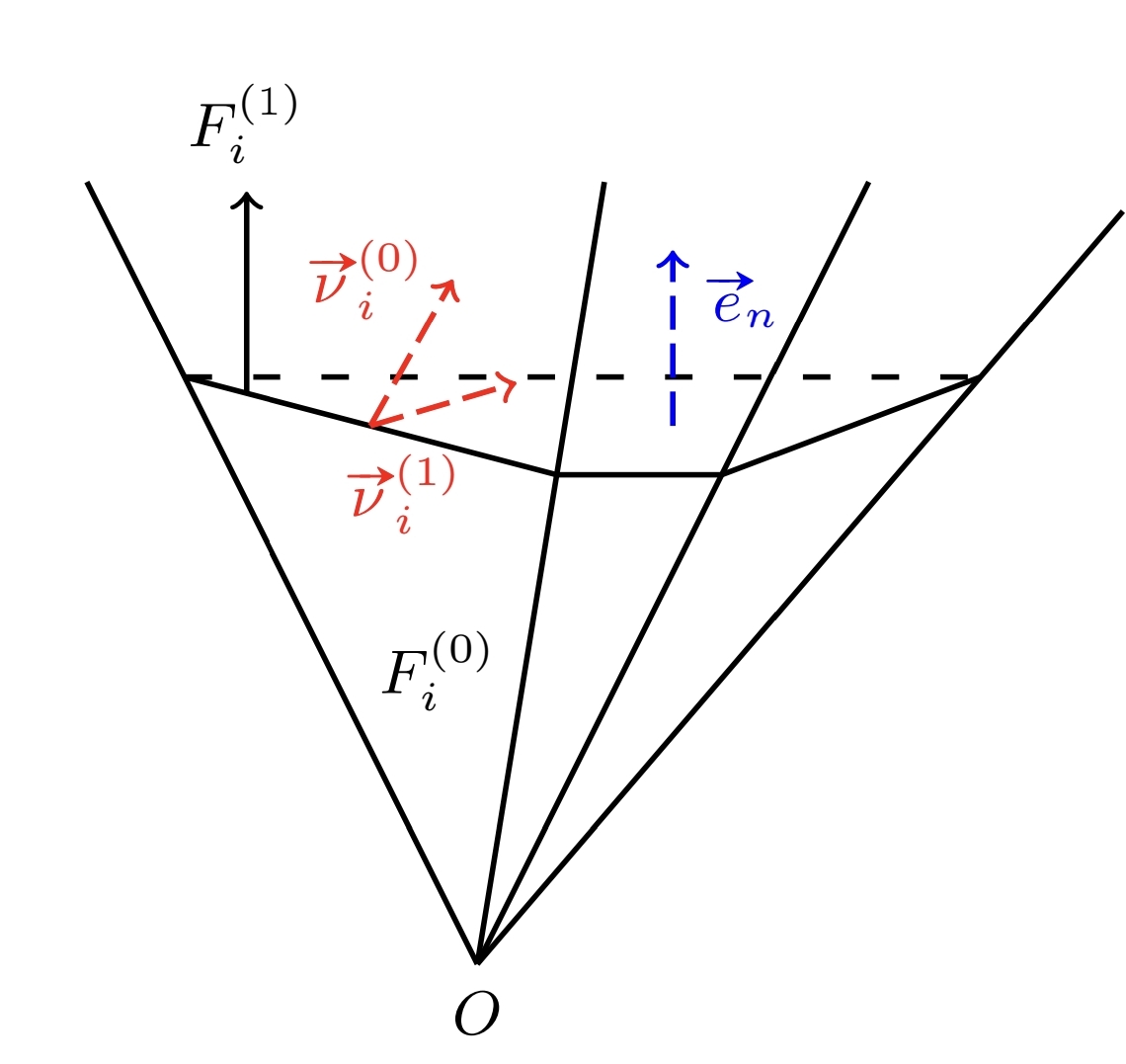}
\caption{}
\label{fig1}
\end{minipage}
\begin{minipage}[t]{0.45\linewidth}
\centering
\includegraphics[width=7 cm]{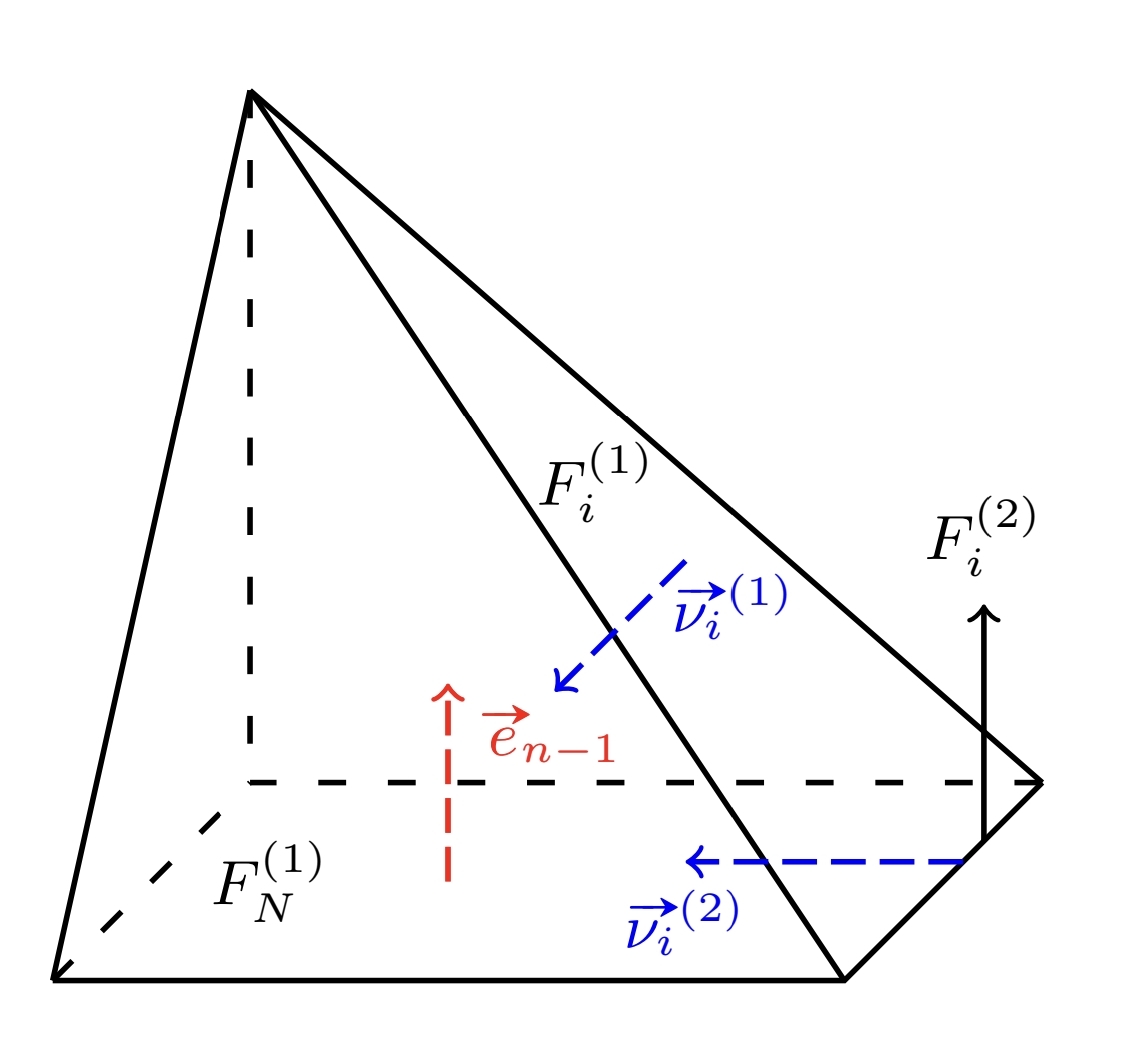}
\caption{}
\label{fig2}
\end{minipage}
\end{figure}
By our assumption, we may assume
\begin{equation*}
D^2 u(0)=I_{n\times n},\quad \vec \nu_i^{(0)}\cdot \vec \nu_j^{(0)}\le 0,\quad \text{if}\quad F_i^{(0)}\sim F_j^{(0)}.
\end{equation*}
Here and in the following, $F_i^{(0)}\sim F_j^{(0)}$ means $F_i^{(0)}$ is adjacent to $F_j^{(0)}$.
Make the following decomposition of $\nu_i^{(0)}$:
\begin{equation*}
\nu_i^{(0)}=(\nu_i^{(0)}\cdot \vec e_n)\vec e_n+(\nu_i^{(0)}\cdot \nu_i^{(1)})\nu_i^{(1)},\quad i=1,\cdots, N.
\end{equation*}
Since $\Omega$ is convex and is locally a graph near $0$, we have
\begin{equation*}
\nu_i^{(0)}\cdot \vec e_n>0,\quad \nu_i^{(0)}\cdot \nu_i^{(1)}>0,\quad i=1,\cdots,N.
\end{equation*}
Then, there holds
\begin{equation}\label{aa-1}
\nu_i^{(1)}\cdot \nu_j^{(1)}=\frac{\nu_i^{(0)}\cdot \nu_j^{(0)}-(\nu_i^{(0)}\cdot \vec e_n)\cdot (\nu_j^{(0)}\cdot \vec e_n)}{(\nu_i^{(0)}\cdot \nu_i^{(1)})(\nu_j^{(0)}\cdot \nu_j^{(1)})}<0,\quad\text{if}\quad F_i^{(1)}\sim F_j^{(1)}.
\end{equation}
In getting the above equality, we also used $F_i^{(0)}\sim F_j^{(0)}\iff F_i^{(1)}\sim F_j^{(1)}$.
\par From the above discussion, we obtain an  convex polytope $\Omega^{(1)}$ in $\mathbb R^{n-1}$ whose $(n-2)$-faces are $F_i^{(1)}$, $i=1,\cdots,N$, $N\ge n+1$. This implies that $\Omega^{(1)}$ is not a simplex in $\mathbb R^{n-1}$.  Then, either $\Omega^{(1)}$ is not simple or is not simplicial if $n\ge 4$(See P67,\cite{Cromwell1997polyhedra} for a proof).  If $\Omega^{(1)}$ is not simple, then we return to the previous discussion. If $\Omega^{(1)}$ is not simplicial.  Without loss of generality, we may assume $F_{N}^{(1)}=\partial \Omega^{(1)}\cap \{x_{n-1}=0\}$ is not a simplex. Let $F_i^{(1)}$ be the adjacent $(n-2)-$faces of $F_{N}^{(1)}$ and $F_i^{(2)}=F_i^{(1)}\cap F_N^{(1)}$, $i=1,\cdots,\widetilde N$. Denote the interior unit normal of $F_i^{(2)}$ relative to $F_N^{(1)}$ by  $\vec \nu_i^{(2)}$(See Figure \ref{fig2}). From \eqref{aa-1} and the convexity of $\Omega^{(1)}$, one has
\begin{equation*}
\vec \nu_i^{(1)}\cdot \vec \nu_j^{(1)}<0,\quad \vec \nu_i^{(1)}\cdot \vec e_{n-1}<0,\quad \vec \nu_i^{(1)}\cdot \vec \nu_i^{(2)}> 0,\quad i=1,\cdots,\widetilde N, \quad F_i^{(1)}\sim F_j^{(1)},
\end{equation*}
there holds
\begin{equation}\label{aa-2}
\nu_i^{(2)}\cdot \nu_j^{(2)}=\frac{\nu_i^{(1)}\cdot \nu_j^{(1)}-(\nu_i^{(1)}\cdot \vec e_{n-1})\cdot (\nu_j^{(1)}\cdot \vec e_{n-1})}{(\nu_i^{(2)}\cdot \nu_i^{(1)})(\nu_j^{(2)}\cdot \nu_j^{(1)})}<0,\quad F_i^{(2)}\sim F_j^{(2)}.
\end{equation}
Overall, we again obtain a convex polytope $\Omega^{(2)}=F_N^{(1)}$ in $\mathbb R^{n-2}$ which is not a simplex. Moreover, the dihedral angles of any two adjacent $(n-3)-$faces of $\Omega^{(2)}$  are acute.
\par Repeating the above procedure finite times, we will get an $l-$polygon in $\mathbb R^2$ with acute angles, $l\ge 4$. This is impossible and proves the ``only if'' part.
\par Now we show for $n=2,3$, the converse is also true. Let $\Omega$ be a simple convex polytope in $\mathbb R^n$. For $n=2$, consider the following boundary value problem
\begin{equation*}
\begin{split}
\det D^2 u& =\frac 12,\quad \text{in}\quad \Omega,\\
u&=\frac 12 |x|^2,\quad \text{on}\quad \partial\Omega.
\end{split}
\end{equation*}
Then, the solution $u\in C^{2,\alpha}(\overline\Omega)$ satisfies $A-$condition.
\par In the following, we focus on the case $n=3$. In this case, $\Gamma_0(\Omega)=\Gamma_{n-3}(\Omega)$ is a discrete point set.
Let $g(t)$ be a $C^2$ convex function satisfying
\begin{equation*}
\begin{cases}
g(t)=-t,\quad t\le \varepsilon_0,\\
g'(t)\le 0,\quad g''(t)\ge 0,\quad t\in (0,2\varepsilon_0),\quad g'(t)<0,\quad t<2\varepsilon_0,\\
g(t)= -\frac{5\varepsilon_0}{4},\quad t\ge 2\varepsilon_0.
\end{cases}
\end{equation*}
Let $p_i^{(0)}\in \Gamma_0(\Omega)$ and $\Omega\subset \{\ell_{p_i^{(0)}}(x)>0\}$, $\{\ell_{p_i^{(0)}}(x)=0\}\cap \overline\Omega=\{ p_i^{(0)}  \}$. Let $P_i^{(0)}$ be a uniformly  convex quadratic polynomial such that
\begin{equation*}
P_i^{(0)}(p_i^{(0)})=|\nabla P_i^{(0)}(p_i^{(0)})|=0,\quad \Theta(P_i^{(0)},p_i^{(0)},\Omega)<\frac \pi 2.
\end{equation*}
Define
\begin{equation*}
u^{(0)}=\sum_{p_i^{(0)}\in \Gamma_0(\Omega)}g(C_0\ell_{p_i^{(0)}}-P_i^{(0)})
\end{equation*}
where $C_0>0$ is chosen such that
\begin{equation*}
[C_0\ell_{p_i^{(0)}}-P_i^{(0)}]|_{\overline{\Omega}}\ge 0,\quad [C_0\ell_{p_i^{(0)}}-P_i^{(0)}]|_{\overline{\Omega}\backslash B_{\delta_0}(p_i^{(0)})}\ge 2\varepsilon_0.
\end{equation*}
Since $\Theta(P_i^{(0)},p_i^{(0)},\Omega)$ is affine invariant, after an affine transformation, we may assume
\begin{equation*}
\begin{split}
p_i^{(0)}=0,\quad \Omega\cap B_r(0)\subseteq(\mathbb R_+)^3 \cap B_r(0),\quad \ell_{p_i^{(0)}}(x)=\sum_{i=1}^3 x_i,\quad P_i^{(0)}(x)=\frac 12 |x|^2,
\end{split}
\end{equation*} with $
\Theta(\frac 12 |x|^2, 0, \Omega)<\frac \pi 2$
for some $r>0$. Then, one has
\begin{equation*}
 g_{x_kx_l}(C_0\ell_{p_i^{(0)}}-P_i^{(0)})=g''(C_0-x_k)(C_0-x_l)-g'\delta_{kl}.
\end{equation*}
This implies
\begin{equation*}
\Theta(u^0,p_i^{(0)},\Omega)<\frac \pi 2,\quad \forall p_i^{(0)}\in  \Gamma_{0}(\Omega).
\end{equation*}
This implies $u^{(0)}$ satisfies strong $A-$condition at $\Gamma_0(\Omega)$.
Let
\begin{equation*}
\hat u=u^0+\lambda_0 \chi |x|^2
\end{equation*}
where $\lambda_0>0$ is small and $\chi$ is a cut-off function such that
\begin{equation*}
\chi=0,\quad \text{near}\quad \Gamma_0(\Omega),\quad \chi\equiv 1,\quad \text{away from }\Gamma_0(\Omega).
\end{equation*}

Let $\underline u$ be the solution of the following equation:
\begin{equation*}
\begin{split}
&\det D^2 \underline u=\det D^2 \hat u-\delta \chi,\quad \text{in}\quad \Omega,\\
&\underline u=\varphi,\quad \text{on}\quad \partial\Omega
\end{split}
\end{equation*}
for some $\delta>0$ small enough. Then, by Theorem \ref{mainthm1} and Lemma \ref{cd2-pC2-1}, one knows $\underline u$ is the desired function for $n=3$.
\end{proof}

As a consequence, we have the following corollary.
\begin{cor}\label{cor-notsimple-3d}
Let $\Omega\subseteq\mathbb R^n$ be a bounded convex polytope and $P$ is not simple.
Suppose that  \eqref{intro-1} admits a globally $C^2$, convex, sub-solution $\underline u\in C^2(\overline{P})$. In addition, $f,\varphi$ satisfy condition \eqref{cond-f-varphi}. Then, \eqref{intro-1} admits a convex solution $u\in C^2(\overline{\Omega})$ if and only if $u=\underline u$.
\end{cor}

\section{$C^{2,\alpha}$ regularity up to $\partial P$}\label{sec-C2alpha}
In this section, we study the $C^{2,\alpha}$ regularity of solutions to \eqref{intro-1} in bounded convex polytopes and prove Theorem \ref{mainthm2}.

{\bf Proof of Theorem \ref{mainthm2}.}
\par Step 1. {\it $\forall K\subset\subset \Gamma_{n-2}(\Omega)\backslash\Gamma_{n-3}(\Omega)$, $u$ is $C^{2,\alpha}$ in a neighbourhood of $K$. } Assume $0\in K$ and $u(0)=|\nabla u(0)|=0$. After an affine transformation, we can assume that
\begin{equation*}
\Omega\cap B_{\rho}(0)=((\mathbb R_+)^2\times \mathbb R^{n-2})\cap B_{\rho}(0),
\end{equation*}
and also
\begin{equation*}
u(x)=P_0(x)+o(|x|^2), \text{ in } \Omega\cap B_{\rho}(0),
\end{equation*}
where $P_0(x)$ is a second order homogeneous polynomial.
From Theorem \ref{thm-2+growth}, we know
\begin{equation*}
|u(x)-P_0(x)|\le \sigma |x|^{2+\alpha}, \text{ in } \Omega\cap B_{\rho}(0)
\end{equation*}
for some $\sigma>0,\alpha\in (0,\beta]$ and small $\rho>0$.
Denote
\begin{equation*}
u_r(x)=\frac{u(rx)}{r^2},\quad f_r(x)=f(rx),\quad \varphi_r(x)=\frac{\varphi(rx)}{r^2},\quad r\in (0, \rho/4).
\end{equation*}
Then, there holds
\begin{equation*}
\|u_r-P_0\|_{L^\infty(
([0,4]^{2}\setminus [0,1]^{2} )     \times (-4,4)^{n-2})}+\|\varphi_r-P_0\|_{C^{2,\alpha}([-4,4]^{n})}\le \sigma r^{\alpha}
\end{equation*}
for some positive constant $\sigma>0$.
Notice that $\|u_r(x)\|_{C^{2,\alpha}(([0,4]^{2}\setminus [0,1)^{2} )   \times [-4,4]^{n-2})}\le C_0.$ Hence, $u_r$ solves the following uniformly elliptic equation with H\"older continuous coefficients
\begin{equation*}
\det D^2 u_r-\det D^2 P_0=f(rx)-f(0)=O(r^\alpha), \text{ in }  ([0,4]^{2}\setminus [0,1)^{2} )  \times (-4,4)^{n-2}.
\end{equation*}
By Schauder estimate for linear uniformly elliptic equation, one has
\begin{equation*}
\|u_r-P_0\|_{C^{2,\alpha}( ([0,3]^{2}\setminus [0,2)^{2} )    \times [-3,3]^{n-2})}\le C\sigma r^{\alpha}.
\end{equation*}
We can do the same arguments for any point in $K$. This implies $u$ is $C^{2,\alpha}$ in a neighbourhood of $K$.
\par Step 2. {\it $\forall K\subset\subset \Gamma_{n-k}(\Omega)\backslash\Gamma_{n-k-1}(\Omega)$, $u$ is $C^{2,\alpha'}$ in a neighbourhood of $K$, $3\le k\le n$.}
After an affine transformation, we can assume that
\begin{equation*}
\Omega\cap B_{\rho}(0)=((\mathbb R_+)^k\times \mathbb R^{n-k})\cap B_{\rho}(0).
\end{equation*}
Following the same notations as in Step 1, by  Theorem \ref{thm-2+growth}, we also know
\begin{equation*}
|u(x)-P_0(x)|\le \sigma |x|^{2+\alpha}, \text{ in } \Omega\cap B_{\tilde \rho}(0).
\end{equation*}
Hence, there holds
\begin{equation*}
\|u_r-P_0\|_{L^\infty(((0,4)^{k}\backslash (0,1)^k)\times (-4,4)^{n-k})}+\|\varphi_r-P_0\|_{C^{2,\alpha}([-4,4]^{n})}\le \sigma r^{\alpha}.
\end{equation*}
We take $k=3$ to explain the strategy. Define
\begin{equation*}
Q_\delta=([0,\delta]^2\times [1,4])\cup ([0,\delta]\times [1,4]\times [0,\delta])\cup ([1,4]\times [0,\delta]^2).
\end{equation*}
Repeating the arguments in Step 1 yields that
\begin{equation*}
\|u_r\|_{C^{2,\alpha}(Q_\delta\times [-4,4]^{n-3})}\le C
\end{equation*}
for some $\delta>0$ small. Boundary estimate and interior estimate also give
 \begin{equation*}
 \|u_r\|_{C^{2,\alpha}(((0,4)^{k}\backslash ((0,1)^k \cup Q_\delta))\times [-4,4]^{n-3})}\le C.
 \end{equation*}
 By Lemma \ref{lemma-interpolation}, one knows
 \begin{equation*}
 \|u_r-P_0\|_{C^{2,\frac{\alpha}2}(((0,4]^{3}\backslash (0,1)^3)\times [-4,4]^{n-3})}\le \widetilde Cr^{\frac{\alpha^3}{2(1+\alpha)^{2}}}
 \end{equation*}
We can do the same arguments for any point in $K$. This implies $u$ is $C^{2,\alpha'}$ in a neighbourhood of $K$ for $\alpha'=\frac{\alpha^3}{2(1+\alpha)^{2}}$.
\par For general $3\le k\le n$, it is enough to repeat the above arguments finite times. This ends the proof of Theorem \ref{mainthm2}.
\qed

\smallskip
\smallskip


\begin{lemma}\label{lem-expansion}
For $\mu\in(\frac{1}{3},\frac{1}{2})$, let $\Omega=\big(V_{\mu}\cap B_{2}((0,0))\big)\times (-2,2)$ and
$u\in C^{2,\alpha} (\overline{\Omega}),\alpha\in(0,1)$ be a solution of
\begin{equation}
\begin{split}
&\det D^2 u=1\quad \text{in}\quad \Omega,\\
&u=\frac{|x|^2}{2} \quad \text{on}\quad \partial\Omega\cap \partial(V_{\mu}\times\mathbb{R})
\end{split}
\end{equation}
satisfying
\begin{equation}\label{D2u-x3}
D^2u=I_{3\times 3} \quad\quad \text{on    }\partial\Omega\cap\{x|x_1=x_2=0\}.
\end{equation}
Then, for $x_3\in[-3/2,3/2]$, there exists a function $c(x_3)$ such that
\begin{equation}\label{u-expansion1}
\left|u-\frac{|x|^2}{2}-c(x_3)r'^{\frac{1}{\mu}}\sin\frac{\theta'}{\mu}\right|\leq   C r'^3,
\end{equation}
for some positive constant $C>0$ depending only on $\mu$, $\alpha$ and $\|u\|_{C^{2,\alpha}(\overline{\Omega} ) }$. Here $(r',\theta')$ is the polar coordinate of  $(x_1,x_2)$.
Moreover, if $c(x_3)\equiv 0$ for $x_3\in[-3/2,3/2]$, then there holds
\begin{equation}\label{u-expansion2}
|u_{ij}(x)-\delta_{ij}|\leq C r', \quad\quad \text{in    } \big( V_{\mu}\cap B_{\frac{1}{100}}((0,0))\big)\times [-1,1].
\end{equation}
\end{lemma}

\begin{proof}
We first prove \eqref{u-expansion1}. It is enough to prove that
\begin{equation}\label{u-expansion3}
|u(x_1,x_2,0)-\frac{x_1^2+x_2^2}{2}-c(0)r'^{\frac{1}{\mu}}\sin\frac{\theta'}{\mu}|\leq C r'^3.
\end{equation}

Set $v=u-\frac{|x|^2}{2}$. Then
\begin{equation*}\Delta v= 1+\Delta v-\det(I_{3\times3}+ D^2 v):= F_0(v).\end{equation*}

For any $x\in\mathbb R^n\setminus\{0\}$,
set $(t,\theta)\in \mathbb R\times \mathbb S^{n-1}$ by
\begin{equation*}
t=-\ln|x|,\quad \theta=\frac{x}{|x|}.\end{equation*}
Denote $\partial^2_t+\Delta_{\theta}$ by $\widetilde{\Delta}$. In the following, we use the superscript $\	\widetilde{}\	$ to distinguish the representation of the same function in different coordinate systems.
Then,
\begin{equation}\label{eq-in-t}\widetilde{\Delta}\widetilde{v}=e^{-2t}\widetilde{F_0(v)  }.\end{equation}

Let $\Sigma=(V_{\mu}\times\mathbb{R})\cap\mathbb{S}^{2}$, $\lambda_1 (\Sigma)=\frac{1}{\mu}\in(2,3)$ be the first Dirichlet eigenvalue of $-\Delta_{\theta}$ on $\Sigma$ and  $\phi_\Sigma=r'^{\frac{1}{\mu}}\sin\frac{\theta'}{\mu}|_{\mathbb{S}^{2}}$ be the corresponding eigenfunction. Also, we know  $\lambda_2 (\Sigma)>3$.
\eqref{D2u-x3} and \eqref{eq-in-t} imply
\begin{equation*}
 \|\widetilde{v}\|_{C^{2, \alpha}([\tau-1,\tau+1)\times\Sigma)  }\leq Ce^{-(2+\alpha) \tau},\quad \tau\ge 10.
\end{equation*}

Following the proof of Theorem 2.1 in \cite{HuangShen2023}, one gets that there exists a constant  $c_1$ such that
\begin{equation}\label{v-expansion}
   \Big|\widetilde{v}-c_1  e^{-\frac{1}{\mu}t}\phi_\Sigma\Big |
     \leq C e^{-3t}\quad\text{in } (100,\infty)\times\overline{\Sigma}.
\end{equation}
Hence,
\begin{equation*}
   \Big|v-c_1 |x|^{\frac{1}{\mu}} \phi_\Sigma \Big |
     \leq C |x|^3\quad\text{in } (100,\infty)\times\overline{\Sigma}.
\end{equation*}
This implies \eqref{u-expansion3} for $x_3=0$.

For any $x\in \big( V_{\mu}\cap B_{\frac{1}{100}}((0,0))\big)\times [-1,1]$, denote
\begin{equation*}
u_{\lambda}(y)=\frac{u((0,0,x_3)+\lambda y)-l_x(y)}{\lambda^2},\quad \varphi_{\lambda}(y)=\frac 12 |y|^2
\end{equation*}
where $\lambda=r'>0$ and $l_x(y)=u((0,0,x_3))+\lambda \nabla u((0,0,x_3))\cdot y$.
Then, $u_{\lambda}$ solves the following uniformly elliptic equation with $C^{\alpha}$ coefficients
\begin{equation*}
\det D^2 u_{\lambda}-\det D^2\varphi_{\lambda} =0, \text{ in }  \big( V_{\mu}\cap[ B_{2}((0,0))\setminus B_{1/2}((0,0))]\big)\times (-2,2).
\end{equation*}
By Schauder's estimate for linear uniformly elliptic equation, one has
\begin{equation*}
|\partial_{ij} u- \delta_{ij}| \le Cr'.
\end{equation*}
This ends the proof of present lemma.
\end{proof}

Now we can construct example to show that the condition (C3) in Theorem \ref{mainthm2} can't be dropped.
\begin{lemma}\label{exm-sta:label}
Let $\Omega=\{(x_1,x_2,x_3)|x_1>0,x_2>0,10>x_3>2(x_1+x_2)\}$ and
$\varphi=\frac{|x|^2}{2}+\lambda_0\chi x_1x_2$ where $\lambda_0>0$ is a small constant and $\chi$ is a cut-off function such that
\begin{equation*}
\chi=1,\quad \text{in}\quad B_{1/2}((0,0,10)), \quad \chi\equiv 0,\quad \text{in    } B_{1}^{c}((0,0,10)).
\end{equation*}
Then, there exists $f\in C^{0,1}(\overline{\Omega})$ satisfying $\det D^2 \varphi>f$ in $\Gamma_1(\Omega)\backslash \Gamma_0(\Omega)$ such that \eqref{intro-1} admits a solution $u\in C^2(\overline{\Omega})$ and $u\notin C^{2,\alpha}(\overline{\Omega})$ for any $\alpha\in (0,1)$.
\end{lemma}
\begin{remark}\label{rem-a1:label}
It can be checked directly that  $\varphi$ satisfies the $A-$condition on $\Gamma_{0}(\Omega)$ and the strong $A-$condition on $\Gamma_{0}(\Omega)\setminus \{0\}$.
\end{remark}
\begin{proof}
For any $f\in C^{0,1}(\overline{\Omega})$ satisfying $0<f<\det D^2 \varphi$ in $\Omega$ and $f=\det D^2 \varphi$ on $\Gamma_0(\Omega)$, it follows from Theorem \ref{mainthm1} that there exists a solution $ u\in C^2(\overline{\Omega})$ solving \eqref{intro-1}.  By the proof of Theorem \ref{mainthm2} and Remark \ref{rem-a1:label}, for any compact set
$K\subseteq \overline{\Omega}\setminus \{0\}$, one also has $u\in C^{2,\alpha_K}(K)$ for some $\alpha_K\in (0,1)$ depending on $K$.

Now we begin the construction of our example.
Set
\begin{equation*}
g(t)=\begin{cases}
1-2t,\quad t\in [0,1/4],\\
1/2,\quad t\in [1/4,3/4],\\
2-2t,\quad t\in [3/4,1],
\end{cases}
\end{equation*}
and
\begin{equation*}
\widetilde{g}(t)=\begin{cases}
1-t,\quad t\in [0,1/2],\\
t,\quad t\in [1/2,3/4],\\
3-3t,\quad t\in [3/4,1].
\end{cases}
\end{equation*}
Then, $g\leq \widetilde{g}$ and $g(1/2)=\widetilde{g}(1/2)$.
Let $G\in C^{0,1}([0,1/2])$ be given by
\begin{equation*}
G(t)=\begin{cases}
1 ,\quad t=0,\\
1-\frac{1}{2^{k}}+\frac{1}{2^{k+1}}g\big(2^{k+1}(t-\frac{1}{2^{k+1}})\big),\	 t\in(\frac{1}{2^{k+1}},\frac{1}{2^{k}}], \	k\in \mathbb Z^+,
\end{cases}
\end{equation*}
and $\widetilde{G}\in C^{0,1}([0,1/2])$ be given by
\begin{equation*}
\widetilde{G}(t)=\begin{cases}
1,\quad t=0,\\
1-\frac{1}{2^{k}}+\frac{1}{2^{k+1}}\widetilde{g}\big(2^{k+1}(t-\frac{1}{2^{k+1}})\big),
\	 t\in(\frac{1}{2^{k+1}},\frac{1}{2^{k}}], \	k\in \mathbb Z^+.
\end{cases}
\end{equation*}
Then, $G\leq\widetilde{G}$ and $G(\frac{3}{2^{k+2}})=\widetilde{G}(\frac{3}{2^{k+2}})=1-\frac{3}{2^{k+2}}$ for any positive integer $k$.

Let $F\in C^{0,1}(\overline{\Omega})$ be a positive function satisfying:
\begin{equation*}
\begin{cases}
0<F< \det D^2 \varphi \quad \text{ in }\Omega, \\
F= \det D^2 \varphi\quad \text{ on each vertex of }\Omega,\\
F(x_1,x_2,x_3)=G(x_3) \quad \text{ near }0.
\end{cases}
\end{equation*}
Let $U\in C^2(\overline{\Omega})$ be the convex solution of
\begin{equation}\label{C2alpha-1}
\begin{split}
\det D^2 U=&F(x),\quad \text{in}\quad \Omega,\\
U=&\varphi(x),\quad \text{on}\quad \partial  \Omega.
\end{split}
\end{equation}
Then, $\partial_{12}U(0,0,x_3)=\sqrt{1-G(x_3)}=\frac{\sqrt{3}}{2^{k/2+1}}$,  $x_3\in [\frac{5}{2^{k+3}},\frac{7}{2^{k+3}}] $.

Set $\mu_k=\frac{\arccos \sqrt{1-(1-\frac{3}{2^{k+2}})} }{\pi}= \frac{\arccos\frac{ \sqrt{3}}{2^{k/2+1}} }{\pi}$. Then for $k$ large, $ \mu_k\in(1/3,1/2)$ and $\mu_k\rightarrow \frac{1}{2}$ as $k \rightarrow \infty$.
Consider
\begin{equation*}
W_k(x)=\frac{(U-\ell_k)\bigg(A_{\mu_k}^{-1}\big( (0,0, \frac{3}{2^{k+2}})+\frac{1}{2^{k+4}} x  \big )\bigg)}{\frac{1}{2^{2k+8}}} ,
\end{equation*}
where $\ell_k(x)$ is the supporting plane of $U$ at $(0,0, \frac{3}{2^{k+2}})$, $A_{\mu_k}$ is given by \eqref{def-Amu} for $\mu=\mu_k$.
Then, $W_k\in C^{2,\alpha_{k}} (\overline{\big(V_{\mu_k}\bigcap B_{2}((0,0))\big)\times (-2,2) })$ for some $\alpha_k\in(0,1)$ solves
\begin{equation}
\begin{split}
&\det D^2 W_k=1\quad \text{in}\quad \big(V_{\mu_k}\bigcap B_{2}((0,0))\big)\times (-2,2),\\
&W_k=\frac{|x|^2}{2} \quad \text{on}\quad \partial\bigg( \big(V_{\mu_k}\bigcap B_{2}((0,0))\big)\times (-2,2) \bigg)\bigcap \partial(V_{\mu_k}\times\mathbb{R})
\end{split}
\end{equation}
with
\begin{equation}\label{D2w-x3}
D^2W_k=I_{3\times 3} \quad\quad \text{on    }\partial\bigg( \big(V_{\mu_k}\bigcap B_{2}((0,0))\big)\times (-2,2) \bigg) \cap\{x|x_1=x_2=0\}.
\end{equation}
Then, Lemma \ref{lem-expansion} implies
\begin{equation}\label{Uexpansion}
|W_k(x_1,x_2,x_3)- \frac{|x|^2}{2} -c_k(x_3)r'^{\frac{1}{\mu_k}}\sin\frac{\theta'}{\mu_k}|=O(r'^3)
\end{equation} for $x_3\in [-3/2,3/2]$, where $(r',\theta')$ is the coordinates of the point $(x_1,x_2)$ in the polar coordinates.

Note that $\mu_k\rightarrow \frac{1}{2}$ when $k\rightarrow\infty$.
Hence, if there exist infinitely many $k$ and $x_3\in [-1,1]$ such that $c_k(x_3)\neq 0$, then for any fixed $\alpha\in(0,1)$, $U\notin C^{2,\alpha}$ near $0$. Otherwise, by Lemma \ref{lem-expansion},
$|D^2 W_{k}(x_1,x_2,x_3)-I_{3\times 3}|\leq Cr'$. Hence, in the following, we assume $c_k(x_3)\equiv 0$, $x_3\in [-1,1]$.

Let $\widetilde{U}\in C^2(\overline{\Omega})$ be the convex solution of
\begin{equation}\label{C2alpha-2}
\begin{split}
\det D^2 \widetilde{U}=&\widetilde{F}(x),\quad \text{in}\quad \Omega,\\
\widetilde{U}=&\varphi(x),\quad \text{on}\quad \partial  \Omega,
\end{split}
\end{equation}
where $\widetilde{F}\in C^{0,1}(\overline{\Omega})$ is a positive function satisfying:
\begin{equation*}
\begin{cases}
F\leq \widetilde{F}< \det D^2 \varphi \quad \text{ in }\Omega, \\
\widetilde{F}= \det D^2 \varphi\quad \text{ on each vertex of }\Omega,\\
\widetilde{F}(x_1,x_2,x_3)=\widetilde{G}(x_3)+\frac{x_1+x_2}{10} \quad \text{ near }0.
\end{cases}
\end{equation*}

We now prove  for any fixed $\alpha\in(0,1)$, $\widetilde{U}\notin C^{2,\alpha}$ for $x$ sufficient close to $0$.
Consider
\begin{equation*}
\widetilde{W_k}(x)=\frac{(\widetilde{U}-\widetilde{\ell}_k)\bigg(A_{\mu_k}^{-1}\big( (0,0, \frac{3}{2^{k+2}})+\frac{1}{2^{k+4}} x  \big )\bigg)}{\frac{1}{2^{2k+8}} },
\end{equation*}
where $\widetilde{\ell_k}(x)$ is the supporting plane of $\widetilde{U}$ at $(0,0, \frac{3}{2^{k+2}})$. Then, $\widetilde{W_k}$ solves
\begin{equation}
\begin{split}
&\det D^2\widetilde{ W_k}=1+h_k(x)\quad \text{in}\quad \big(V_{\mu_k}\bigcap B_{2}((0,0))\big)\times (-2,2),\\
&\widetilde{W_k}=\frac{|x|^2}{2} \quad \text{on}\quad \partial\bigg( \big(V_{\mu_k}\bigcap B_{2}((0,0))\big)\times (-2,2) \bigg)\bigcap \partial(V_{\mu_k}\times\mathbb{R})
\end{split}
\end{equation}
for $h_k(x)\in C^{0,1}$ satisfying
\begin{equation*}
\frac{1}{C}\frac{|x|}{2^{k+4}}\le h_k(x)\le \frac{C|x|}{2^{k+4}}
\end{equation*}
and  $D^2\widetilde W_k(0)=I_{3\times 3}.$ Then,
$\widetilde{W_k}\in C^{2,\beta_k} (\overline{\big(V_{\mu_k}\bigcap B_{2}((0,0))\big)\times (-2,2) })$ for some $\beta_k\in(0,1)$.

By Hopf's lemma, for small $\delta>0$, there exists $\epsilon_{\delta}>0$ such that, $\forall 0<\epsilon\leq \epsilon_{\delta}$, there holds
\begin{equation*}
\widetilde{W_{k}} \leq W_{k}-\epsilon r'^{\frac{1}{\mu_k}}\sin\frac{\theta'}{\mu_k} \quad\text{on }\big(V_{\mu_k}\bigcap \partial B_{\delta}((0,0))\big)\times [-1,1].
\end{equation*}
Define
\begin{equation*}
W'=W_{k}-\epsilon r'^{\frac{1}{\mu_k}}\sin\frac{\theta'}{\mu_k}+\delta^{2}\epsilon x_{3}^4.
\end{equation*}
Then, $\widetilde{W_k}\leq W'$ on $\partial \left((V_{\mu_k}\bigcap  B_{\delta}((0,0)))\times [-1,1]\right)$.
A direct computation also yields that
\begin{equation*}
\det D^2 W' \leq \det D^2\widetilde{W_{k}},  \quad \text{in}\quad  \big(V_{\mu_k}\cap  B_{\delta}((0,0))\big)\times [-1,1]
\end{equation*}
when $\delta>0$ and $\epsilon>0$ are small enough.

By the maximum principle, we have
\begin{equation*}
\widetilde{W_k}\leq W'_k,\quad \text{in}
\quad \big(V_{\mu_k}\cap  B_{\delta}((0,0))\big)\times [-1,1].
\end{equation*}
 In particular, we have
\begin{equation*}
\widetilde{W_{k}}((x_1,x_2,0))\leq W_{k}((x_1,x_2,0))-\epsilon r'^{\frac{1}{\mu_k}}\sin\frac{\theta'}{\mu_k}.
\end{equation*}
Therefore, $\widetilde{U} \notin C^{2,\alpha}(\overline{\Omega})$ for any $\alpha\in(0,1)$.
\end{proof}
\smallskip
\smallskip
In application, it is rather difficult to construct a desired $C^2$ sub-solution $\underline u$. In the following, we give a checkable criterion to guarantee the global $C^{2,\alpha}-$estimates.
\begin{definition}
A function $\varphi\in C^1(\overline{\Omega})$ is called {\bf uniformly convex on $\partial\Omega$} if and only if there holds
\begin{equation*}
\varphi(x)-\varphi(y)-\nabla \varphi(y)\cdot (x-y)\ge c|x-y|^2,\quad \forall x,y\in \partial\Omega
\end{equation*}
for some positive  constant $c$.
\end{definition}

\begin{theorem}\label{thm6.3}
Let $\Omega$  be a bounded convex $n-$polytope.  Let $0<f\in C^{\beta}(\overline{\Omega})$ and $\varphi\in C^{2,\beta}(\overline \Omega)$ be two given functions such that   $\varphi $ satisfies conditions (C1),(C2). Moreover, $\varphi$ is uniformly convex on $\partial\Omega$.
Then, the following conclusions hold true.
\begin{itemize}
	\item[1.] $n=2$. There exists a small positive constant $\varepsilon>0$ such that \eqref{intro-1} admits a unique solution $u\in C^{2,\alpha}(\overline{\Omega})$ provided that $f\le \varepsilon$.
	\item[2.] $n=3$. Let $p_i$, $i=0,\cdots,N$ be the vertices of $\Omega$. There exist small constants $\delta,\varepsilon(\delta),\omega(\delta)>0$ such that if \begin{equation*}
\begin{split}
 f(x)&\le  f(y)+\omega(\delta),\quad \forall y\in \Gamma_{0}(\Omega),\quad \forall x\in B_{\delta}(y)\cap\overline \Omega,\\
f(x)&\le \varepsilon(\delta),\quad x\notin B_{\delta^6}(\Gamma_{0}(\Omega)),
\end{split}
\end{equation*}
\end{itemize}
then, \eqref{intro-1} adimits a unique solution $u\in C^{2,\alpha}(\overline{\Omega})$ for some $\alpha\in (0,1)$.
\end{theorem}
\begin{proof}
Step 1. $n=2$. Up to an affine transformation, we may assume $0$ is a vertex of $\Omega$ and the tangent cone of $\Omega$ at $0$ is $(\mathbb R_+)^2$. Without loss of generality, we also assume
\begin{equation*}
\varphi_{x_1x_1}(0)=\varphi_{x_2x_2}(0)=1,\quad \frac{1}C |x|^2\le \varphi(x)\le C|x|^2,\quad x\in \partial\Omega
\end{equation*}
for some positive constant $C>1$ depending only on the uniformly convexity of $\varphi$ on $\partial\Omega$. Hence, if $f\le \varepsilon$ for $\varepsilon>0$ small, then by standard maximum principle we know
\begin{equation*}
u\ge \frac{1}{C}|x|^2,\quad x\in \Omega.
\end{equation*}
Then, by Lemma \ref{lemma-liou3}, there holds
\begin{equation*}
u(x)=\frac 12 |x|^2+\sqrt{1-f(0)}x_1x_2+o(|x|^2),\quad \text{near}\quad 0.
\end{equation*}
From this, by rescaling argument, we know $u\in C^2(\overline{\Omega})$. Taking $u$ as the $C^2(\overline{\Omega})$ sub-solution, we conclude from Theorem \ref{mainthm2} that  $u\in C^{2,\alpha}(\overline{\Omega})$.

Step 2. $n=3$. Let  $p_0=0$  be a vertex of $\Omega$. Since $\varphi \in C^{2,\beta}(\overline{\Omega})$, there exists a uniformly convex  quadratic polynomial $P_0(x)$ such that
\begin{equation*}
|\varphi(x)-P_0(x)|\le \|\varphi\|_{C^{2,\beta}(\partial\Omega)}|x|^{2+\beta},\quad x\in \partial \Omega,\quad \det D^2 P_0=f(0).
\end{equation*}
After an affine transformation, we may assume
\begin{equation*}
P_0(x)=\frac 12|x|^2+\sqrt{1-f(0)}x_1x_2
\end{equation*}
and the tangent cone of $\Omega$ at $0$ contains $x_1=0$ and $x_2=0$ as its $2-$faces.
Then, we can choose  $\kappa=2\omega(\delta)$
 small such that
$$P_{0,\kappa}(x)=(1/2-\kappa)|x|^2+(\sqrt{1-f(0)}-\frac{4\kappa}{\sqrt{1-f(0)}})x_1x_2$$
satisfies $\det D^2 P_{0,\kappa}\ge f(0)+2\omega(\delta)$.
Then, there holds
\begin{equation*}
\begin{split}
\varphi\ge & P_0- \|\varphi\|_{C^{2,\beta}(\partial\Omega)}|x|^{2+\beta}\\
\ge & P_{0,\kappa}+\omega(\delta)|x|^2-9\|\varphi\|_{C^{2,\beta}(\partial\Omega)}\delta^\beta|x|^{2}\\
\ge & P_{0,\kappa}+\frac 12 \omega(\delta)|x|^2,
\end{split} \quad \text{in}\quad B_{2\delta(0)}\cap \partial\Omega,
\end{equation*}
provided $\omega(\delta)>>\delta^\beta$.
Let $h(t)$ be the solution of the following equation
\begin{equation*}
h''(t)=\min\left(1,2\left(1-\frac{t}{\delta^3}\right)^+\right),\quad h(0)=h'(0)=0.
\end{equation*}
Suppose $P_{0,\kappa}(x)=\frac 12 |Ax|^2$ for some positive matrix $A$.
Introduce a $C^2$-smooth convex function
\begin{equation}
H(x)=h(|Ax|).
\end{equation}
Then, from the construction, we know $H(x)$ is convex and satisfies
\begin{equation*}
H(x)=
\begin{cases}=P_{0,\kappa}(x),\quad |Ax|\le \frac{\delta^3}{2},\\
\le P_{0,\kappa}(x),\quad  \frac{\delta^3}{2} \le |Ax|\le \delta^3\\
\le \delta^3|Ax|-\delta^6/2\le P_{0,\kappa}(x),\quad |Ax|\ge \delta^3.
\end{cases}
\end{equation*}
For the remaining vertices $p_i$, $i=1,\cdots,N$ of $\Omega$, we set
\begin{equation*}
g_1(x)=g(C_0 (x-p_1)\cdot \nu_1-2|f|_{L^\infty(\Omega)}^{\frac 1n}|x|^2)+\frac{5}{4}\varepsilon_0
\end{equation*}
where $g(t)$ is given by Proposition \ref{prop5.1:label},  $\nu_1$ is the inner normal of $\Omega$ at $p_1$, and $\varepsilon_0^{-1}\sim \delta^{-3}$. $g_i$, $i=2,\cdots,N$ can be constructed similarly.
Let $\chi(x)\in [0,1]$ be a cut-off function such that $\chi\equiv 1$ in $\Omega\backslash(\cup_{i=0}^N B_{\delta^6}(p_i))$ and $\chi\equiv 0$ in $\Omega \cap (\cup_{i=0}^N B_{\frac 12\delta^6}(p_i))$. Denote
\begin{equation*}
u_0=H(x)+\sum_{i=1}^N g_i+ \delta^6 \chi\cdot |x|^2.
\end{equation*}
\par We first compare the boundary values of $u_0,u$.
On $\partial\Omega \cap B_\delta(0)$, we have
\begin{equation*}
u_0(x)=H(x)+ \delta^6 \chi\cdot |x|^2\le P_{0,\kappa}(x)+\delta^6 |x|^2\le \varphi(x).
\end{equation*}
On $\partial\Omega \backslash B_\delta(0)$, we have
\begin{equation*}
u_0(x)\le \delta^3 |Ax|+C(\delta^6 +\delta^3)\le \frac{\delta |x|}{2C}\le \varphi(x).
\end{equation*}
\par Then, we  compare the values of $\det D^2 u_0,\det D^2 u$.
In $\Omega\cap B_{\delta}(0)$, we have
\begin{equation*}
\begin{split}
\det D^2 u_0\ge& \det D^2 (P_{0,\kappa}+\delta^6 \chi |x|^2)\\
\ge& \det D^2 P_{0,\kappa}-C\delta^6\\
\ge& f(x)=\det D^2 u.
\end{split}
\end{equation*}
In $\Omega\cap \{g_i<\varepsilon_0\}$, we have
\begin{equation*}
\det D^2 u_0\ge \det D^2 (g_i+\delta^6 \chi |x|^2)\ge 2|f|_{L^\infty(\Omega)}-C\delta^6 \ge \det D^2 u.
\end{equation*}
In the remaining domain, we have
\begin{equation*}
\det D^2 u_0\ge C\delta^{18}\ge \varepsilon \geq \det D^2 u.
\end{equation*}

\par Hence, by maximum principle, we have $u\ge u_0$. This implies
\begin{equation*}
\frac{u_\lambda}{\lambda^2}\ge P_{0,\kappa},\quad \lambda\rightarrow 0
\end{equation*}
up to a subsequence. By Lemma \ref{lem-liou-perturb}, one knows $\frac{u_\lambda}{\lambda^2}\rightarrow \frac{1}{2}|x|^2+\sqrt{1-f(0)}x_1x_2$. Let $x^k$ be a sequence of points such that $\lim_{k \rightarrow +\infty}x^k=0$. By interior and boundary regularity, we only consider the case that
\begin{equation*}
\frac{x^k}{|x^k|}\rightarrow \bar y\in \mathbb S^2\cap \{x_1=x_2=0\}.
\end{equation*}
Then, we consider the point $y^k=(0,0,x^k_3)$. After a subtraction of a linear function and affine transformation, we may assume $u(y^k)=\nabla u(y^k)=0$. We can repeat the above construction to get a similar conclusion that
\begin{equation*}
\frac{u(\lambda x+y^k)}{\lambda^2}\rightarrow \frac{1}{2}|x|^2+\sqrt{1-f(0)}x_1x_2,\quad \text{in}\quad (\mathbb R^+)^2\times \mathbb R
\end{equation*}
by using Lemma \ref{lemma-liou3}. This implies $u$ is $C^2$ at $0$.
\par For any point $p=(0,0,a)$ along the edge $x_1=x_2=0$, we can consider two cases.
\\ (1). $a\in (0,\delta/2)$. Then, the above construction of sub-solution also holds, i.e., we can prove
\begin{equation*}
u(x)\ge P_{0,\kappa}(x-p)+u(p)+\nabla u(p)\cdot (x-p),\quad \text{near}\quad p.
\end{equation*}
Then, by Remark \ref{rem-face-2-a} and  Lemma \ref{lemma-liou3}, we know
$$\frac{u(\lambda x+p)-u(p)-\lambda x\cdot \nabla u(p)}{\lambda^2}\rightarrow \sum_{i,j=1}^{n}q_{ij}(p)x_{i}x_{j},$$
where
\begin{equation*}
q_{ij}(0,0,a)=\varphi_{ij}(p),\quad i,j=1,\cdots,n,\quad i+j\ne 3.
\end{equation*}and
$q_{12}(p)$ is the big root of the following univariate quadratic equation
\begin{equation*}
A q_{12}^2(p)+Bq_{12}(p)+C=0
\end{equation*}
which is equivalent to the equation $\det D^2 q=f(p)$.

(2). $a>\frac 12 \delta$.  Then, from the boundary condition, we know
\begin{equation*}
u(x)\ge u(p)+\nabla u(p)\cdot (x-p)+\frac{1}{C}|x-p|^2,\quad \text{on}\quad \partial\Omega.
\end{equation*}
Consider $\psi(x)=u(p)+\nabla u(p)\cdot (x-p)+\frac{1}{2C}|x-p|^2$. Near vertex $p_i$, we set \begin{equation*}
g_i(x)=g(C_0(x-p_i)\cdot \nu_i -2|f|_{L^\infty(\Omega)}|x-p_i|^2)+\frac{5}{4}\varepsilon_0
\end{equation*}
where $g(t)$ is given by Proposition \ref{prop5.1:label} and $\nu_i$ is the inner normal of $\Omega$ at $p_i$. Here $\varepsilon_0\sim\delta^3$. Let
\begin{equation*}
u_p(x)=\psi(x)+\sum_{i=0}^N g_i(x).
\end{equation*}
Then, by assumptions and maximum principle, we know that  $u(x)\ge u_p(x)$ in $\Omega$.  When $\varepsilon$ is sufficiently small,  by Remark \ref{rem-face-2-a} and  Lemma \ref{lemma-liou3},  we know that
\begin{equation*}
\frac{u(\lambda x+p)-u(p)-\lambda x\cdot \nabla u(p)}{\lambda^2}\rightarrow \sum_{i,j=1}^{n}q_{ij}(p)x_{i}x_{j},\quad \text{in}\quad (\mathbb R^+)^2\times \mathbb R.
\end{equation*}
where
\begin{equation*}
q_{ij}(p)=\varphi_{ij}(p),\quad i,j=1,\cdots,n,\quad i+j\ne 3,
\end{equation*}
and
$q_{12}(p)$ is the big root of the following univariate quadratic equation
\begin{equation*}
A q_{12}^2(p)+B q_{12}(p)+C=0
\end{equation*}
which is equivalent to the equation $\det D^2 q=f(p)$.
From (1), (2), we know $u$ is also $C^2$ along the edges.
\end{proof}

\section{Appendix}
In this Appendix, we give the following simple interpolation inequality.
\begin{lemma}\label{lemma-interpolation}
Let $f\in C^{1,\alpha}([0,1])$, $\alpha \in (0,1]$ be such that
\begin{equation*}
\|f\|_{L^{\infty}([0,1])}\leq A,\quad \|f'\|_{C^{0,\alpha}([0,1])}\leq B
\end{equation*}
 for two positive constants $A$ and $B$.
Suppose $A\leq B$.
Then, there holds
$$\|f'\|_{L^{\infty}([0,1])}\leq 6A^{\frac{\alpha}{1+\alpha}}B^{\frac{1}{1+\alpha}},$$
and also
\begin{equation*}
[f]_{C^{\frac{\alpha}{2}}([0,1])}\le (12 B^{\frac 1{1+\alpha}}+B) A^{\frac{\alpha}{2(1+\alpha)}}.
\end{equation*}
\end{lemma}

\begin{proof}
$\forall x\in[0,1]$, by Taylor's expansion, we have
$$f(y)-f(x)=f'(x)(y-x)+[f'(x+\theta(y-x))-f'(x)](y-x)$$
for some $\theta\in(0,1)$. Taking $y\in[0,1]$ such that $|y-x|=\frac{1}{2}\big(\frac{A}{B}\big)^{\frac{1}{1+\alpha}}\leq \frac{1}{2}$, one gets
$$\frac{1}{2}\big(\frac{A}{B}\big)^{\frac{1}{1+\alpha}}|f'(x)|\leq 2A+\frac{A}{2^{1+\alpha} }.$$
Hence, $|f'(x)|\leq 6A^{\frac{\alpha}{1+\alpha}}B^{\frac{1}{1+\alpha}} .$  Consider any two points $x,y\in [0,1]$.
\\ 1. $|x-y|\ge A^{\frac 1{(1+\alpha)}}$. Then
\begin{equation*}
\begin{split}
\frac{|f'(x)-f'(y)|}{|x-y|^\frac{\alpha}{2}}\le 2\frac{|f'|_{L^\infty([0,1])}}{A^{\frac \alpha{2(1+\alpha)}}} \le 12 B^{\frac 1{1+\alpha}}A^{\frac{\alpha}{2(1+\alpha)}}.
\end{split}
\end{equation*}
\\ 2. $|x-y|\le A^{\frac 1{(1+\alpha)}}$. Then
\begin{equation*}
\begin{split}
\frac{|f'(x)-f'(y)|}{|x-y|^\frac{\alpha}{2}}\le A^{\frac \alpha {2(1+\alpha)}}\frac{|f'(x)-f'(y)|}{|x-y|^\alpha} \le B A^{\frac{\alpha}{2(1+\alpha)}}.
\end{split}
\end{equation*}
This ends the proof of present lemma.
\end{proof}

\end{document}